\documentclass[tbtags,reqno]{amsart}
\usepackage{graphicx,color}
\usepackage{tikz}
\usepackage{amsmath,amsfonts,amstext,amsbsy,amsopn,amsthm,amscd,amsxtra,dsfont,amssymb,pb-diagram}
\usepackage{amsmath,amsthm,amsopn,amsfonts,amssymb,epsfig,enumerate}
\usepackage{amsthm,amsopn,amsfonts,amssymb,epsfig,tabmac,mathrsfs,cancel,wasysym}
\usepackage[all]{xy}
\usepackage{tabmac}
\usepackage[active]{srcltx}

\numberwithin{equation}{section}

\makeatletter
\renewcommand{\subsubsection}{\@startsection
{subsubsection} {3} {0mm} {\baselineskip} {-0.5\baselineskip} {\normalfont\normalsize\bfseries}} \makeatother

\def\cal L{{\mathcal L}}

\newcommand{\Aset}{\mathcal A}
\newcommand{\Bset}{\mathcal B}
\newcommand{\Cset}{\mathcal C}

\newcommand{\La}{\Lambda}

\newcommand{\OSS}{{\footnotesize \tableau[scY]{\tf \ocircle}}}
\newcommand{\BSS}{{\footnotesize  \tableau[scY]{ \fl }}}
\newcommand{\WSS}{{\footnotesize \tableau[scY]{ \\   }}}
\newcommand{\BCC}{{\footnotesize \tableau[scY]{\bl \cerclep}}}
\newcommand{\WCC}{{\footnotesize \tableau[scY]{\bl \tcercle{}}}}

\def\AA {\mathfrak A}
\def\BB {\mathfrak B}
\def \part {\vdash}
\def\La {\Lambda}

\def\cd{\circledast}

\newcommand{\cercle}[1]{\ensuremath{\setlength{\unitlength}{1ex}\begin{picture}(2.8,2.8)\put(1.5,1.2){\circle{2.8}\makebox(-5.6,0){#1}}\end{picture}}}
\newcommand{\tcercle}[1]{\ensuremath{\setlength{\unitlength}{1ex}\begin{picture}(2.8,2.8)\put(1.4,1.4){\circle{2.8}\makebox(-5.6,0){#1}}\end{picture}}}

\newcommand{\cerclep}{\ensuremath{\setlength{\unitlength}{1ex}\begin{picture}(2.8,2.8)\put(1.4,1.4){\circle*{2.7}\makebox(-5.6,0)}\end{picture}}}

\theoremstyle{plain}
\newtheorem{definition}{Definition} 
\newtheorem{theorem}[definition]{Theorem}
\newtheorem{example}[definition]{Example}
\newtheorem{lemma}[definition]{Lemma}
\newtheorem{corollary}[definition]{Corollary}
\newtheorem{proposition}[definition]{Proposition}

\newtheorem{remark}[definition]{Remark}

\theoremstyle{definition}

\numberwithin{equation}{section}

\title{Symmetry and Pieri rules for the bisymmetric Macdonald polynomials}
\begin{document}

\author{Manuel Concha}
\address{Instituto de Matem\'aticas, Universidad de
Talca, 2 norte 685, Talca, Chile.}
\email{manuel.concha@utalca.cl}
\author{Luc Lapointe}
\address{Instituto de Matem\'aticas, Universidad de
Talca, 2 norte 685, Talca, Chile.}
\email{llapointe@utalca.cl}

\begin{abstract}
Bisymmetric Macdonald polynomials can be obtained through a process of antisymmetrization and $t$-symmetrization of non-symmetric Macdonald polynomials. Using the double affine Hecke algebra,
we show that the evaluation of the bisymmetric Macdonald polynomials satisfies a symmetry property generalizing that satisfied by the usual Macdonald polynomials.
We then obtain Pieri rules for the bisymmetric Macdonald polynomials where the sums are over certain vertical strips.
\end{abstract}


\thanks{Funding: this work was supported by
the Fondo Nacional de Desarrollo Cient\'{\i}fico y Tecnol\'ogico de Chile (FONDECYT) Regular Grant \#1210688 and the Beca de Doctorado Nacional ANID \#21201319.}

\maketitle

\section{Introduction}
An extension to superspace of the Macdonald polynomials was defined in \cite{BDLM,BDLM2}.
In this new setting, the Macdonald polynomials now depend on the anticommuting variables $\theta_1,\dots,\theta_N$ as well as the usual variables $x_1,\dots,x_N$.  But when focusing on the dominant monomial in the anticommuting variables (of degree $m$, let's say),
the Macdonald polynomials in superspace essentially correspond to non-symmetric Macdonald polynomials whose variables $x_1,\dots,x_m$ (resp. $x_{m+1},\dots,x_N$) are antisymmetrized (resp. $t$-symmetrized).  To be more precise, the Macdonald polynomial in superspace $P_\Lambda(x,\theta;q,t)$, indexed by the superpartition $\Lambda$, is such that
\begin{equation} \label{eqanti}
  P_\Lambda(x,\theta;q,t) \longleftrightarrow \mathcal A_{1,m} \mathcal S^t_{m+1, N} E_\eta(x;q,t)
\end{equation}
where $E_\eta(x;q,t)$ is a non-symmetric Macdonald polynomial, and where 
$\mathcal A_{1,m}$ (resp. $\mathcal S^t_{m+1, N}$) stands for the antisymmetrization (resp. $t$-symmetrization) operator in the variables $x_1,\dots,x_m$ (resp. $x_{m+1},\dots,x_N)$. We note that polynomials obtained through such a
process of antisymmetrization and $t$-symmetrization of non-symmetric Macdonald polynomials
were also considered in \cite{BDF,Ba}, where they were called Macdonald polynomials with prescribed symmetry.

Many of the important features of the Macdonald polynomials have been extended
to superspace. For instance, the duality and the existence of Macdonald operators were shown in \cite{BDLM2}
while explicit formulas for the norm and the evaluation  were provided in \cite{GL}. Most remarkably, an extension of the original Macdonald positivity conjecture was stated in \cite{BDLM2}.
Moreover, conjectural Pieri rules for the Macdonald polynomials in superspace were given in \cite{GJL} (and shown to naturally connect to the 6-vertex model) while a symmetry property for the Macdonald polynomials in superspace was presented  in \cite{ABDLM,B} as a conjecture.

In this article, we will work with a bisymmetric version of the Macdonald polynomials obtained by dividing the polynomials in  \eqref{eqanti}
by a $t$-Vandermonde determinant.  Our
bisymmetric Macdonald polynomial are thus defined (up to a constant) as
\begin{equation} 
\mathcal  P_\Lambda(x;q,t) \propto \frac{1}{\Delta^t_m(x)} \mathcal A_{1,m} \mathcal S^t_{m+1, N} E_\eta(x;q,t)
\end{equation}
Using properties of the double affine Hecke algebra, we first  show that the bisymmetric Macdonald polynomials satisfy two symmetry properties. Indeed, to each of the two natural evaluations in superspace, $u_{\Lambda_0}^-$ and $u_{\Lambda_0}^+$, corresponds a symmetry property:
$$
u_\Omega^-(\tilde{\mathcal P}_{\Lambda}^-)=u_\Lambda^-(\tilde{\mathcal P}_{\Omega}^-)
\qquad {\rm and} \qquad u_\Omega^+(\tilde{\mathcal P}_{\Lambda}^+)=u_\Lambda^+(\tilde{\mathcal P}_{\Omega}^+)
$$
where $\tilde{\mathcal P}_{\Lambda}^-$ and $\tilde{\mathcal P}_{\Lambda}^+$ are both equal to $\mathcal  P_\Lambda(x;q,t)$ up to a (distinct) constant. By relation \eqref{eqanti}, this implies that the Macdonald polynomials in superspace also satisfy the same symmetries, thus proving a conjecture appearing in \cite{ABDLM,B}.  We should mention that the crucial step in the proof of the symmetry is the construction of a new symmetric pairing stemming from the double affine Hecke algebra (see Lemma~\ref{SimGen}).

We then proceed to obtain the Pieri rules associated to the elementary symmetric function $e_{r}(x_{m+1},\dots,x_N) $.  The proof turns out to be  quite long and technical, making heavy use of the properties of the double affine Hecke algebra.   As a  first step, an explicit expression for the action of $ e_{r}(Y_{m+1},\dots,Y_N)$ on any bisymmetric function (times a $t$-Vandermonde determinant) is derived, where the $Y_i$'s are the Cherednik operators.  The symmetries we just described are then used to deduce the Pieri rules:
\begin{equation} \label{Pieri1}
  e_{r}(x_{m+1},\dots,x_N)    {\mathcal P}_{\Lambda}(x;q,t) = \sum_\Omega \left(\frac{u_\Lambda^+(C_{J,\sigma})}{ u_\Lambda^+(\Delta_m^t)} \frac{u_{\Lambda_0^+}(\mathcal P_\Lambda)}{u_{\Lambda_0^+}(\mathcal P_\Omega)} \right)   {\mathcal P}_{\Omega}(x,q,t)  
\end{equation}  
where the sum is over all the superpartitions $\Omega$ such that  $\Omega/\Lambda$ belongs to a certain family of vertical $r$-strips,
called vertical $r$-strips of type I (see Theorem~\ref{theoPieri} and Corollary~\ref{coroPieri}). We stress that all the terms in the coefficients have explicit expressions.

Finally, we obtain the other set of Pieri rules:
\begin{equation} \label{Pieri2}
  e_{r}(x_{1},\dots,x_m)    {\mathcal P}_{\Lambda}(x;q,t) = q^r \sum_\Omega \left(\frac{u_\Lambda^+(D_{J,\sigma})}{ u_\Lambda^+(\Delta_m^t)} \frac{u_{\Lambda_0^+}(\mathcal P_\Lambda)}{u_{\Lambda_0^+}(\mathcal P_\Omega)} \right)   {\mathcal P}_{\Omega}(x,q,t)  
\end{equation}  
where the sum is over all the superpartitions $\Omega$ such that  $\Omega/\Lambda$  belongs to a slightly  different family of vertical $r$-strips,
called vertical $r$-strips of type II (see Theorem~\ref{PieriAlg-} and Corollary~\ref{coroPieri2}). The article being already quite long, it seemed more appropriate to only provide the main steps of the proof (for the full proof see \cite{Con}), which is in any case very similar to the proof of the Pieri rules associated to $e_{r}(x_{m+1},\dots,x_N)$.

The Pieri rules for the Macdonald polynomials in superspace correspond to the action of $e_r(x_1,\dots,x_N)$ on a Macdonald polynomial in superspace.  An application of our results would be to 
deduce those Pieri rules  from \eqref{Pieri1} and \eqref{Pieri2} using the relation \cite{Mac}
$$
e_r(x_1,\dots,x_N) = \sum_{\ell=0}^r e_\ell(x_1,\dots,x_m) e_{r-\ell}(x_{m+1},\dots,x_N)
$$
The conjectured Pieri rules for $e_r(x_1,\dots,x_N)$ involve the Izergin-Korepin determinant related to the partition function of the  6-vertex model. It would be interesting to understand how this determinant originates from the coefficients appearing in \eqref{Pieri1} and \eqref{Pieri2}.
\section{Preliminaries}
\subsection{Double affine Hecke algebra and non-symmetric Macdonald polynomials}
The non-symmetric Macdonald polynomials can be defined as the common eigenfunctions of the Cherednik operators \cite{Che}, which are operators that belong to the double affine Hecke algebra and act on the ring $\mathbb Q(q,t)[x_1,\dots,x_N]$.  We now give the relevant definitions \cite{Mac2,Mar}.  
Let the exchange operator $K_{i,j}$ be such that
$$K_{i,j} f(\dots, x_i,\dots,x_j,\dots)= f(\dots, x_j,\dots,x_i,\dots)$$
We then define the generators $T_i$ of the affine Hecke algebra as
\begin{equation} \label{eqTi}
T_i=t+\frac{tx_i-x_{i+1}}{x_i-x_{i+1}}(K_{i,i+1}-1),\quad i=1,\ldots,N-1,
\end{equation}
and
$$
 {T_0=t+\frac{qtx_N-x_1}{qx_N-x_1}(K_{1,N}\tau_1\tau_N^{-1}-1)}\, ,
$$
where $\tau_i f(x_1,\dots, x_i,\dots,x_N)= f(x_1,\dots, qx_i,\dots,x_N)$ is the $q$-shift operator.
The $T_i$'s satisfy the relations  {($0\leq i\leq N-1$)}:
\begin{align*} &(T_i-t)(T_i+1)=0\nonumber\\
&T_iT_{i+1}T_i=T_{i+1}T_iT_{i+1}\nonumber\\
&T_iT_j=T_jT_i \, ,\quad i-j \neq \pm 1 \mod N
\end{align*}
where the indices are taken modulo $N$.
To define the Cherednik operators, we also need to introduce
the operator $\omega$ defined as:  
$$
\omega=K_{N-1,N}\cdots K_{1,2} \, \tau_1.
$$
We note that $\omega T_i=T_{i-1}\omega$ for $i=2,\dots,N-1$.

We are now in position to define the Cherednik operators:
$$
Y_i=t^{-N+i}T_i\cdots T_{N-1}\omega \bar T_1 \cdots \bar T_{i-1},
$$
where 
$$
\bar T_j:= T_j^{-1}=t^{-1}-1+t^{-1}T_j,
$$
 which follows from the quadratic relation satisfied by the generators
 of the  Hecke algebra.
 The Cherednik operators obey the following  relations:
 \begin{align} \label{tsym1}
T_i \, Y_i&= Y_{i+1}T_i+(t-1)Y_i\nonumber \\
T_i \, Y_{i+1}&= Y_{i}T_i-(t-1)Y_i\nonumber \\
T_i Y_j & = Y_j T_i \quad {\rm if~} j\neq i,i+1.
\end{align}
It can be easily deduced from these relations that
\begin{equation}\label{TYi}
(Y_i+Y_{i+1})T_i= T_i (Y_i+Y_{i+1}) \qquad {\rm and } \qquad (Y_i Y_{i+1}) T_i =
T_i (Y_i Y_{i+1}). 
\end{equation}

An element  $\eta=(\eta_1,\dots,\eta_N)$ of $\mathbb Z_{\geq 0}^{N}$ is called a (weak) composition with $N$ 
parts (or entries).
It will prove convenient to represent a composition by a Young (or Ferrers) diagram.  Given a composition $\eta$ with $N$ parts, let $\eta^+$ be the partition obtained by reordering the entries of $\eta$.  The diagram corresponding to $\eta$ is the Young diagram of $\eta^+$ with an $i$-circle (a circle filled with an $i$) added to the right of the row of size $\eta_i$ (if there are many rows of size $\eta_i$, the circles are ordered from top to bottom in increasing order).  For instance, given $\eta=(0,2,1,3,2,0,2,0,0)$, we have
 $$
\eta \quad \longleftrightarrow  \quad {\tableau[scY]{&& & \bl \cercle{4}\\& & \bl \cercle{2} \\& & \bl \cercle{5}\\ & & \bl \cercle{7} \\ & \bl \cercle{3} \\ \bl \cercle{1} \\ \bl \cercle{6} \\ \bl \cercle{8} \\ \bl \cercle{9}}}
 $$

The Cherednik operators $Y_i$'s commute among each others, $[Y_i,Y_j]=0$,
and can be simultaneously diagonalized. Their eigenfunctions are the
 (monic) non-symmetric Macdonald polynomials (labeled by compositions).
For $x=(x_1,\dots,x_N)$, 
the non-symmetric Macdonald polynomial $E_\eta(x;q,t)$ is 
the 
unique polynomial with rational coefficients in $q$ and $t$ 
that is triangularly related to the monomials (in the Bruhat ordering on compositions) 
\begin{equation} \label{orderE}
E_\eta(x_1,\dots,x_N;q,t)=x^\eta+\sum_{\nu\prec\eta}b_{\eta\nu}(q,t) \, x^\nu
\end{equation}
and that satisfies, for all $i=1,\dots,N$, 
\begin{equation} 
  Y_i E_\eta=\bar \eta_iE_\eta,\qquad\text{where}\qquad  \bar\eta_i =q^{\eta_i}t^{1-r_\eta(i)} \label{eigenvalY}
\end{equation}
  with $r_\eta(i)$ standing for the row (starting from the top) in which the $i$-circle appears in
  the diagram of $\eta$.
 The Bruhat order on compositions is defined as follows:
 $$
   \nu\prec\eta\quad \text{ {iff}}\quad \nu^+<\eta^+\quad \text{or} \quad \nu^+=\eta^+\quad \text{and}\quad w_\eta < w_\nu,
 $$
 where $w_{\eta}$ is the unique permutation of minimal length such 
that $\eta = w_{\eta} \eta^+$ ($w_{\eta}$ permutes the entries of $\eta^+$).
In the Bruhat order on the symmetric group, $w_\eta {<} w_\nu$ iff
$w_{\eta}$ can be obtained as a  {proper} subword  of $w_{\nu}$.  The Cherednik operators have a triangular action on monomials \cite{Mac2}, that is,
\begin{equation} \label{triangY}
Y_i x^\eta =   \bar \eta_i x^\eta + {\rm ~smaller~terms}
\end{equation}
where ``smaller terms'' means that the remaining monomials $x^\nu$  appearing in the expansion are such that $\nu \prec \eta$.

Finally, for an interval $[a,b]=\{a,a+1,\dots,b\} \subseteq \{1,\dots,N\} $, we introduce the $t$-symmetrization and antisymmetrization operators
\begin{equation} \label{antisym} 
\mathcal S_{a,b}^t = \sum_{\sigma \in \mathfrak S_{a,b}} T_\sigma \qquad {\rm and} \qquad
\mathcal A_{a,b} = \sum_{\sigma \in  \mathfrak S_{a,b}} (-1)^{\ell(\sigma)}K_\sigma
\end{equation}
where $ \mathfrak S_{a,b}$ stands for the permutation group of $[a,b]$, and where $K_\sigma$
is such that
$$
K_{\sigma} f (x_1,\dots,x_N) =f(x_{\sigma(1)}, \dots, x_{\sigma(N)})
$$
For simplicity, we use $[m]$ for the interval $[1,m]$, and $\mathcal A_m$ for
$\mathcal A_{1,m}$.

\subsection{Bisymmetric Macdonald polynomials}

A partition $\lambda=(\lambda_1,\lambda_2,\dots)$ of degree $|\lambda|$
is a vector of non-negative integers such that
$\lambda_i \geq \lambda_{i+1}$ for $i=1,2,\dots$ and such that
$\sum_i \lambda_i=|\lambda|$.  The length $\ell(\lambda)$
of $\lambda$ is the number of non-zero entries of $\lambda$.
Each partition $\lambda$ has an associated Ferrers diagram
with $\lambda_i$ lattice squares in the $i^{th}$ row,
from the top to bottom. Any lattice square in the Ferrers diagram
is called a cell (or simply a square), where the cell $(i,j)$ is in the $i$th row and $j$th
column of the diagram \cite{Mac}.

We will be concerned with the ring of bisymmetric functions $\mathscr R_{m,N} = \mathbb Q[x_1,\dots,x_N]^{ \mathfrak S_m \times  \mathfrak S_{m+1,N}}$. Bases of $\mathscr R_{m,N}$ are naturally indexed by pairs of partitions $(\lambda,\mu)$, where $\lambda$ (resp. $\mu$) is a partition whose length is at most $m$ (resp. $N-m$). We will adopt the language of symmetric functions in superspace \cite{BDLM2} and consider the bijection
$$
(\lambda,\mu) \longleftrightarrow (\Lambda^a;\Lambda^s) := (\lambda+\delta_m;\mu)
$$
where $\delta_m=(m-1,m-2,\dots,0)$. The superpartition $\Lambda= (\Lambda^a,\Lambda^s)$ thus consists of a partition $\Lambda^a$ with $m$ non-repeated entries (one of them possibly equal to zero) and a usual partition $\Lambda^s$ whose length is not larger than $N-m$.

We will let $\Lambda^*$ be the partition obtained by reordering the entries of the concatenation of $\Lambda^a$ and $\Lambda^s$. Similarly, we will let  $\Lambda^{\circledast}$ be the 
 partition obtained by reordering the entries of the concatenation of $\Lambda^a+(1^m)$ and $\Lambda^s$. 
A diagrammatic representation of $\La$ is given by 
the Ferrers diagram of $\La^*$ with circles added in the extra cells corresponding to $\Lambda^{\circledast}$.
For instance, if $\La=(\Lambda^a;\Lambda^s)
=(3,1,0;2,1)$,  we have $\Lambda^\circledast=(4,2,2,1,1)$ and $\Lambda^*
=(3,2,1,1)$, so that 
{\small
\begin{equation*} \label{exdia}
     \La^\cd=\quad{\tableau[scY]{&&&\\&\\&\\\\ \\ }}  \qquad
         \La^*=\quad{\tableau[scY]{&&\\&\\ \\ \\ }} \qquad
 \Longrightarrow\qquad      \La=\quad {\tableau[scY]{&&&\bl\tcercle{}\\&\\&\bl\tcercle{}\\ \\
     \bl\tcercle{}}}
 \end{equation*}}
\hspace{-0.3cm}

Given $\Lambda=(\Lambda^a;\Lambda^s)$, we consider that $\Lambda$ is the composition $\Lambda=(\Lambda_1,\Lambda_2,\dots,\Lambda_N)$ where $(\Lambda_1,\dots,\Lambda_m)=\Lambda^a$
and $(\Lambda_{m+1},\dots,\Lambda_N)=\Lambda^s$.  Define the composition
$\eta_{\Lambda}=(\Lambda_1,\dots,\Lambda_m,\Lambda_N,\Lambda_{N-1},\dots,\Lambda_{m+1})$.

We let $\Delta_I(x)$, for  $I=\{i_1,i_2,\dots,i_m\}$ with $i_1 <i_2 < \cdots < i_m$, stand for the  Vandermonde determinant $\prod_{1 \leq j <k \leq m}(x_{i_j}-x_{i_k})$. In the case where $I=[1,m]$, we use the shorter notation $\Delta_{m}(x)$ instead of $\Delta_{[1,m]}(x)$.  The $t$-Vandermonde determinant is $\Delta_m^t(x)=\prod_{1 \leq j <k \leq m}(tx_{j}-x_{k})$.

The bisymmetric Macdonald polynomial indexed by the superpartition $\Lambda$ is  then defined as
\begin{equation}
\mathcal P_\Lambda(x_1,\dots,x_N;q,t)= \frac{c_{\Lambda}(t)}{\Delta^t_m(x) } \, \mathcal A_{m} \mathcal S^t_{m,N } \, E_{\eta_{\Lambda}}(x_1,\dots,x_N;q,t)  
\end{equation}
with the normalization constant $c_{\Lambda}(t)$ such that
\begin{equation} \label{normalization}
\frac{1}{c_{\Lambda}(t)}=    \left( \prod_{i \geq 0} [n_{\Lambda^s}(i)]_{t^{-1}}! \right)
t^{(N-m)(N-m-1)/2} 
\end{equation}
where $n_{\Lambda^s}(i)$ is the number of entries in $\Lambda^s$ that are equal to $i$, and where
$$
[k]_q=\frac{(1-q)(1-q^2)\cdots (1-q^k)}{(1-q)^k}
$$
We observe that the normalization constant $c_{\Lambda}(t)$ is chosen such that the coefficient of
$$x_1^{\lambda_1} \cdots x_m^{\lambda_m}  x_{m+1}^{\mu_1} \cdots x_{N}^{\mu_{N-m}}$$
in $\mathcal P_\Lambda(x;q,t)$ is equal to 1, where $(\lambda,\mu) \longleftrightarrow \Lambda$.  This somewhat non trivial observation is a consequence of the connection with Macdonald polynomials in superspace which we are about to explain.

The bisymmetric Macdonald polynomials are in correspondence with the Macdonald polynomials in superspace.  This is seen in the following way.
The Macdonald polynomials in superspace provide a basis of the ring of symmetric polynomials in superspace 
$\mathbb Q[x_1,\dots,x_N;\theta_1,\dots,\theta_N]^{ \mathfrak S_N}$, where the variables $\theta_1,\dots,\theta_N$ are Grassmannian variables (that is such that $\theta_i \theta_j = -\theta_j \theta_i$ if $i \neq j$ and such that $\theta_i^2=0$), and where the action of $\mathfrak S_N$ is the diagonal one.  Any 
polynomial in superspace $F(x;\theta)$ can be written as
\begin{equation} \label{eqsP}
F(x;\theta) = \sum_{I \subseteq \{ 1,\dots,N \}; |I|=m} \theta_I \Delta_I(x) f_I(x)
\end{equation}
where, for $I=\{i_1,\dots,i_m\}$ with $i_1 < i_2 < \cdots < i_m$, we have
$\theta_I=\theta_{i_1}\cdots \theta_{i_m}$. Observe that by symmetry
a polynomial in superspace is completely determined by its coefficient $f_{\{1,\dots,m\}}(x)$, and moreover, that $f_{\{1,\dots,m\}}(x)$ needs to be bisymmetric.  The Macdonald polynomial in superspace $\mathcal P_\Lambda(x_1,\dots,x_N;\theta_1,\dots,\theta_N;q,t)$ is then simply the unique symmetric polynomial in superspace
of the form \eqref{eqsP} such that
$f_{\{1,\dots,m\}}(x)= \mathcal P_\Lambda(x_1,\dots,x_N;q,t) \in \mathscr R_{m,N}$. As such, $\mathscr R_{m,N}$ is isomorphic, as a vector space, to the subspace of $\mathbb Q[x_1,\dots,x_N;\theta_1,\dots,\theta_N]^{ \mathfrak S_N}$ made out of polynomials of homogeneous degree $m$ in the Grassmannian variables. We can thus immediately infer that the bisymmetric Macdonald polynomials form a basis of $\mathscr R_{m,N}$.
\section{Evaluations and symmetry}
The element $w$ of the symmetric group $\mathfrak S_N$ acts on a vector $(v_1,\dots,v_N) \in \mathbb Z^N$ as $w (v_1,\dots,v_N) =(v_{w^{-1}(1)},\dots,v_{w^{-1}(N)})$.

Let $w$ be the minimal length permutation such that  $w\Lambda = \Lambda^{*}$.  The positive evaluation $u^{+}_{\Lambda}$  is defined on any $f(x) \in \mathscr R_{m,N}$ as
\begin{equation} \label{poseval}
    u^+_{\Lambda}(f(x_1,\dots,x_N))=  f(q^{\Lambda_{w(1)}^{\circledast}}t^{1-w(1)},\dots,q^{\Lambda_{w(N)}^{\circledast}}t^{1-w(N)})
\end{equation}
while the negative evaluation $u^-_{\Lambda}$ is defined as 
\begin{equation} \label{negeval}
 u^-_{\Lambda}(f(x_1,\dots,x_N))=  f(q^{-\Lambda_{w(1)}^{*}}t^{w(1)-1},\dots,q^{-\Lambda_{w(N)}^{*}}t^{w(N)-1})
\end{equation}
\begin{remark} \label{remark1}
  In the case of $\Lambda_0=(\delta_m;\emptyset)$, where  $\delta_m=(m-1,m-2,\dots,0)$, the negative evaluation corresponds to an evaluation considered in \cite{BDLM,GL}.
To be more precise, if the symmetric function in superspace is $F(x,\theta)$ as given in \eqref{eqsP},  
then $u_{\Lambda_0}^-(f_{\{1,\dots,m\}}(x))=\varepsilon^m_{t^{N-m},q,t}\bigl(F(x,\theta)\bigr)$ in the language of \cite{GL}. 
\end{remark}
It turns out that we can use other permutations than the one of  minimal length when taking the evaluations.  We use the notation $\Lambda+(1^m)$ for the vector
$(\Lambda_1+1,\dots,\Lambda_m+1,\Lambda_{m+1},\dots,\Lambda_N)$.
\begin{lemma} 
  Let $\sigma$ be any permutation such that  $\sigma\Lambda = \Lambda^{*}$  and  $\sigma(\Lambda + (1^m)) = \Lambda^{\circledast}$.
 Then, when computing $u_\Lambda^+(f)$ and $u_\Lambda^-(f)$ for a bisymmetric function $f$, the permutation  $\sigma$ can be used  in \eqref{poseval} and \eqref{negeval} instead of the minimal permutation $w$.
  That is, when computing $u_\Lambda^+(f)$ and $u_\Lambda^-(f)$ for a bisymmetric function $f$, we have in this case that
$$
    u^+_{\Lambda}(f(x_1,\dots,x_N))=  f(q^{\Lambda_{\sigma(1)}^{\circledast}}t^{1-\sigma(1)},\dots,q^{\Lambda_{\sigma(N)}^{\circledast}}t^{1-\sigma(N)})
    $$
    and
$$
    u^-_{\Lambda}(f(x_1,\dots,x_N))=  f(q^{-\Lambda_{\sigma(1)}^{*}}t^{\sigma(1)-1},\dots,q^{-\Lambda_{\sigma(N)}^{*}}t^{\sigma(N)-1})
$$
\end{lemma}
\begin{proof}
  We first prove that if $w$ is the minimal permutation such that  $w\Lambda = \Lambda^{*}$ then $w(\Lambda + (1^m)) = \Lambda^{\circledast}$.  Suppose that $w$ is such a minimal permutation and suppose that  $\Lambda_i= \Lambda_j$ with $i \leq m$ and $j > m$.  Then, by minimality, $w^{-1}(i) < w^{-1}(j)$ which means that
$\Lambda_i+1$ occurs to the left of $\Lambda_j$ in $w(\Lambda + (1^m))$. We then  deduce immediately that $w(\Lambda + (1^m))=\Lambda^{\circledast}$.

  Now, let $\sigma$ be such that $\sigma\Lambda = \Lambda^{*}$  and  $\sigma(\Lambda + (1^m)) = \Lambda^{\circledast}$.  As we have just seen, the minimal permutation $w$ is also such that  $w\Lambda = \Lambda^{*}$  and  $w(\Lambda + (1^m)) = \Lambda^{\circledast}$.  Hence,  $w^{-1} \sigma$ acts as the identity on $\Lambda$ and $\Lambda+(1^m)$, which  means in particular that $w^{-1} \sigma \in \mathfrak S_m \times \mathfrak S_{m+1,N}$. Since $f$ is bisymmetric, we thus have $u_\Lambda^+ (K_{w^{-1} \sigma} f)=u_\Lambda^+ (f)$. Hence
  $$
u^+_{\Lambda}(f(x_{w^{-1}\sigma(1)},\dots,x_{w^{-1}\sigma(N)}))=  f(q^{\Lambda_{\sigma(1)}^{\circledast}}t^{1-\sigma(1)},\dots,q^{\Lambda_{\sigma(N)}^{\circledast}}t^{1-\sigma(N)})
  $$
which is equivalent to
\eqref{poseval} after performing the transformation $i \mapsto \sigma^{-1}w(i)$.
The proof in the case of $u^-_{\Lambda}$ is identical.
\end{proof}

We will say that  $(\Lambda,\sigma)$ \textbf{generates a superevaluation} whenever $\Lambda$ is a superpartition such that
\begin{enumerate}
\item $\sigma\Lambda = \Lambda^{*}$
\item $\sigma(\Lambda + (1^m)) = \Lambda^{\circledast}$
\end{enumerate}

\begin{lemma}\label{AutVal} If $f$ is a bisymmetric function then
\begin{equation} \label{aut1}
    f(Y^{-1}) \Delta_m^t(x) \mathcal  P_{\Lambda}(x;q,t) = u_\Lambda^-(f)  \Delta_m^t(x) \mathcal  P_{\Lambda}(x;q,t)
\end{equation}
Moreover, if $g(x_{m+1},\dots,x_N)$ is symmetric in the variables $x_{m+1},\dots,x_N$ then
\begin{equation} \label{aut2}
    g(Y_{m+1},\dots,Y_N) \Delta_m^t(x) \mathcal  P_{\Lambda}(x;q,t) = u_\Lambda^+(g)  \Delta_m^t(x) \mathcal  P_{\Lambda}(x;q,t)
\end{equation}
while if $g(x_{1},\dots,x_m)$ is symmetric in the variables $x_{1},\dots,x_m$  and of homogeneous degree $d$ then
\begin{equation} \label{aut3}
    g(Y_1,\dots,Y_m) \Delta_m^t(x) \mathcal  P_{\Lambda}(x;q,t) = q^{-d} u_\Lambda^+(g)  \Delta_m^t(x) \mathcal  P_{\Lambda}(x;q,t)
\end{equation}
\end{lemma}
\begin{proof}
We first prove \eqref{aut1}. We have that $Y_i^{-1}E_{\eta} = \bar \eta_i^{-1} E_{\eta}$,
where we recall that $\bar \eta_i= q^{\eta_i} t^{1-r_\eta(i)}$. Using the fact that $f$ is bisymmetric, we then have
\begin{equation*}
\begin{array}{ll}
  f(Y_1^{-1},\dots,Y_N^{-1}) \Delta_m^t(x) \mathcal  P_\Lambda(x;q,t) &= \displaystyle f(Y_1^{-1},\dots,Y_N^{-1})  {c_\Lambda(t)}   \mathcal A_{1,m}^t \mathcal S^t_{m+1,,N}  E_{\eta}  \\
  &= \displaystyle  {c_\Lambda(t)}   \mathcal A_{1,m}^t \mathcal S^t_{m+1,,N} f(Y_1^{-1},\dots,Y_N^{-1}) E_{\eta}  \\
   &= \displaystyle  {c_\Lambda(t)}    \mathcal A_{1,m}^t \mathcal S^t_{m+1,,N} f(\bar \eta_1^{-1},\dots, \bar \eta_N^{-1}) E_{\eta}  \\
     &=   f(\bar \eta_1^{-1},\dots, \bar \eta_N^{-1})  \Delta_m^t(x) \mathcal  P_\Lambda(x;q,t)
\end{array}
\end{equation*}
It thus only remains to show that the specialization
$x_i = \overline{\eta}^{-1}_i$ corresponds to the negative evaluation.  We have that $r_\eta(i)=w(i)$, where $w$ is the minimal permutation such that $w\eta=\Lambda^*$. Therefore, $x_i=  \overline{\eta}^{-1}_i=q^{-\eta_i}t^{w(i)-1}$.
By definition, we also have that $\Lambda_{i}^*=\eta_{w^{-1}(i)}$ or equivalently, that $\Lambda_{w(i)}^*=\eta_{i}$.  Hence, the specialization $x_i = \overline{\eta}^{-1}_i$ amounts to  $x_{i} =q^{-\Lambda^*_{w(i)}}t^{w(i)-1}$, as wanted.

As for the proof of \eqref{aut2} and \eqref{aut3}, observe that 
$w^{-1} \Lambda^\circledast = \Lambda + (1^m)$ and $w^{-1} \Lambda^* = \Lambda$ imply that $\Lambda^\circledast_{w(i)}=\Lambda^*_{w(i)}$ for $i\in \{m+1,\dots,N\}$ while   $\Lambda^\circledast_{w(i)}=\Lambda^*_{w(i)}+1$ for $i\in \{1,\dots,m\}$.
Proceeding as in the proof of \eqref{aut1}, we get straightforwardly that
the specialization is at $x_i=\overline{\eta}_i$ instead of $x_i=\overline{\eta}^{-1}_i$.  We can then immediately deduce that
\eqref{aut2} and \eqref{aut3} hold from our previous observation.    
\end{proof}

\subsection{The double affine Hecke algebra and a symmetric pairing}
We will introduce in this subsection a symmetric pairing associated to the evaluation $u_{\Lambda_0}^-$ that generalizes the symmetric pairing in the double affine Hecke algebra.  We first explain  the symmetric pairing in the double affine Hecke algebra.

The  double affine Hecke algebra has a natural basis (over $\mathbb Q(q,t)$) given by the elements of  the form (see for instance (4.7.5) in \cite{Mac2})
 \begin{equation} \label{basis}
x^\eta T_w Y^{-\gamma}
 \end{equation}  
for all $\eta,\gamma \in \mathbb Z^N$ and all permutations $w \in \mathfrak S_N$.
The map
$\varphi$ defined by \cite{Mac2} 
\begin{equation} \label{anti}
\varphi \left(x^\eta T_w Y^{-\gamma} \right) = x^\gamma T_{w^{-1}} Y^{-\eta} 
 \end{equation}  
is an anti-automorphism. Notice that we have in particular that
\begin{equation*}
\varphi(x^\eta)= Y^{-\eta} \qquad {\rm and } \qquad \varphi(Y^{-\gamma})=x^\gamma
\end{equation*}  
The evaluation map $\Theta$ is then defined as
\begin{equation}
\Theta (a) = u_{\emptyset}^-(a \cdot 1) 
\end{equation}
where $u_{\emptyset}^-(f(x_1,\dots,x_N)=f(1,t^{1},\dots,t^{N-1})$ for any Laurent polynomials $f(x)$.  For instance, using
$ f(Y^{-1}) \cdot 1=u_\emptyset^-(f)$ and $T_w \cdot 1=t^{\ell(w)}$,
we have
\begin{align*}
  \Theta \bigl(g(x) T_w f(Y^{-1})\bigr)& =  u_{\emptyset}^-\bigl(g(x) T_w f(Y^{-1}) \cdot 1 \bigr) & \\ & =
  u_{\emptyset}^-(f)  u_{\emptyset}^- \bigl(g(x)  T_w \cdot 1 \bigr)\\
  &=  t^{\ell(w)} u_{\emptyset}^-(f)  u_{\emptyset}^- \bigl(g  \bigr)
\end{align*}
Using the fact that $\ell(w)=\ell(w^{-1})$, we see that
\begin{equation*}
  \Theta \bigl(g(x) T_w f(Y^{-1})\bigr)=
  \Theta \bigl(f(x) T_{w^{-1}} g(Y^{-1})\bigr)=  \Theta \circ \varphi \bigl(g(x) T_w f(Y^{-1})\bigr)
\end{equation*}  
Since the  basis \eqref{basis} is given by elements of the form $g(x) T_w f(Y^{-1})$, we thus have established that
\begin{equation}
\Theta=\Theta \circ \varphi
\end{equation}  
For all Laurent polynomials $f(x),g(x)$, let the pairing $[f,g]$ be defined as
\begin{equation} \label{pairing}
  [f,g]=\Theta\bigl(f(Y^{-1}) g(x)\bigr)
\end{equation}
  Using the previous relation and the fact that $\varphi$ is an anti-automorphism, we immediately get the symmetry of the pairing
\begin{equation*}
[f,g]= \Theta\bigl(f(Y^{-1}) g(x)\bigr)=  \Theta \circ \varphi \bigl(f(Y^{-1}) g(x)\bigr) =\Theta\bigl(g(Y^{-1}) f(x)\bigr) =[g,f]
\end{equation*}

Now, consider the evaluation $u_{\Lambda_0}^-(f)$ given explicitly as
$$
u_{\Lambda_0}^-(f(x_1,x_2,\dots,x_{m},x_{m+1},\dots, x_N))= f(q^{1-m},q^{2-m}t,\dots,q^0 t^{m-1}, t^m,\dots,t^{N-1})
$$
on any Laurent polynomial $f$, where $\Lambda_0$ is such as defined in Remark~\ref{remark1}.

Our goal is to define 
a pairing associated to the evaluation $u_{\Lambda_0}^-$. 
We first need to extend the map  $\Theta$.  Let $\Theta_m$ be such that
 \begin{equation}
\Theta_m (a)= u_{\Lambda_0}^- \bigl( a \cdot E_{\delta_m}(x;q,t)\bigr ) 
\end{equation}
where  $E_{\delta_m}(x;q,t)$ is the non-symmetric Macdonald polynomial indexed by the composition
\begin{equation*}
\delta_m=(m-1,m-2,\dots,1,0,0,\dots,0)
\end{equation*}
Observe that $f(Y^{-1}) \cdot E_{\delta_m}(x;q,t)= u_{\Lambda_0}^-(f)  E_{\delta_m}(x;q,t)$.
\begin{lemma}  Let $\varphi$ be the anti-automorphism defined in \eqref{anti}.
We have that
\begin{equation}
 \Theta_m = \Theta_m \circ \varphi
 \end{equation}
\end{lemma}
\begin{proof}
We first show that
\begin{equation} \label{Thetam}
\Theta_m (T_w) = \Theta_m \circ \varphi (T_w)= \Theta_m (T_{w^{-1}})
\end{equation}  
for any permutation $w \in \mathfrak S_N$. 
Let $F_w(x)= T_w E_{\delta_m}(x;q,t)\cdot 1$.  As this is a polynomial in $x_1,\dots,x_N$, we can consider the quantity
\begin{equation} \label{eqTw}
  \Theta\bigl( F_w(Y^{-1}) E_{\delta_m}(x;q,t) \bigr) = u_{\Lambda_0}^-(F_w) u_\emptyset^-( E_{\delta_m}) = \Theta_m (T_w) u_\emptyset^-(E_{\delta_m}) 
\end{equation}
since  $u_{\Lambda_0}^-(F_w)= u_{\Lambda_0}^-(T_w \cdot  E_{\delta_m}(x;q,t))= \Theta_m (T_w)$.
From  $\Theta = \Theta \circ \varphi$, we also have
\begin{equation*}
 \Theta\bigl( F_w(Y^{-1}) E_{\delta_m}(x;q,t) \bigr)= \Theta\bigl( E_{\delta_m}(Y^{-1};q,t) F_w(x)  \bigr)= \Theta\bigl( E_{\delta_m}(Y^{-1};q,t) T_w E_{\delta_m}(x)  \bigr) 
\end{equation*}  
Using again $\Theta = \Theta \circ \varphi$, we can then replace $w$ by $w^{-1}$ in the term on the right to get
\begin{equation*}
 \Theta\bigl( F_w(Y^{-1}) E_{\delta_m}(x;q,t) \bigr)=  \Theta\bigl( E_{\delta_m}(Y^{-1};q,t) T_{w^{-1}} E_{\delta_m}(x)  \bigr) = \Theta\bigl( E_{\delta_m}(Y^{-1};q,t) F_{w^{-1}}(x)  \bigr)
\end{equation*}  
Using  $\Theta = \Theta \circ \varphi$ one last time, we can transform the term on the right to obtain
\begin{equation} \label{eqTwinv}
 \Theta\bigl( F_w(Y^{-1}) E_{\delta_m}(x;q,t) \bigr)= \Theta\bigl( F_{w^{-1}}(Y^{-1}) E_{\delta_m}(x;q,t) \bigr) = \Theta_m (T_{w^{-1}}) u_\emptyset^-(E_{\delta_m}) 
\end{equation}
Comparing \eqref{eqTw} and \eqref{eqTwinv}, we can thus conclude that
\eqref{Thetam} holds given that $u_{\Lambda_0}^- (E_{\delta_m})$ is not equal to 0 (it can be deduced easily from the fact that $E_\eta(x;q,1)=x^\eta$ for any $\eta$).

We can now prove that $ \Theta_m = \Theta_m \circ \varphi$ holds in general.
Recall that any element of the affine Hecke algebra can be written in the form $f(x) T_w g(Y^{-1})$,
where $f(x),g(x)$ are Laurent polynomials.
On the one hand, we have
\begin{equation*}
  \Theta_m \bigl(f(x) T_w g(Y^{-1}) \bigr)= u_{\Lambda_0}^-(g)
  \Theta_m \bigl(f(x) T_w  \bigr) =
u_{\Lambda_0}^-(g) u_{\Lambda_0}^-(f)
\Theta_m ( T_w )
\end{equation*}
while on the other hand, we have
\begin{equation*}
  \Theta_m \circ \varphi\bigl(f(x) T_w g(Y^{-1}) \bigr)=  \Theta_m \bigl(g(x) T_{w^{-1}} f(Y^{-1}) \bigr)=
  u_{\Lambda_0}^-(f) u_{\Lambda_0}^-(g) \Theta_m ( T_{w^{-1}} )
\end{equation*}
Since we have previously established that $\Theta_m ( T_{w^{-1}} )= \Theta_m ( T_{w} )$, we conclude from the previous two equations that  $ \Theta_m = \Theta_m \circ \varphi$.
\end{proof}

We now define our new pairing.  For any Laurent polynomials $f,g$ symmetric in the variables $x_1,\dots,x_m$, let
\begin{equation}
[f,g]_m = u_{\Lambda_0}^- \bigl(f(Y^{-1}) g(x)  \Delta_m^t(x) \bigr)  
\end{equation}
This new pairing is again symmetric.
\begin{lemma} \label{SimGen} If $f$ and $g$ are two Laurent polynomials that are symmetric in the variables $x_1, \ldots, x_m$, then
  $$[f,g]_m = [g,f]_m$$
\end{lemma}
\begin{proof}
Let $\mathcal A^t_m$ be the $t$-antisymmetrizer in the first $m$ variables
\begin{equation}
\mathcal A^t_m = \sum_{\sigma \in \mathfrak S_m} (-1/t)^{\ell(\sigma)} T_\sigma
\end{equation}
We have that $T_i \mathcal A^t_m  =-  \mathcal A^t_m$ for any $i=1,\dots,m-1$. Hence, for every polynomial $f(x)$, we get that
$\mathcal A^t_m f(x)=\Delta_m^t(x) g(x)$, where $g(x)$ is a polynomial  symmetric in $x_1,\dots,x_m$. In particular, by degree consideration, we have that
$$
\Delta_m^t(x) = c_m(q,t) \mathcal A^t_m E_{\delta_m}(x;q,t)
$$
for some non-zero constant $c_m(q,t)$  (at $t=1$,  the r.h.s. produces the usual Vandermonde determinant, so $c_m(q,1)=1\neq 0$). It is also immediate that
 $\mathcal A^t_m  \mathcal A^t_m = d_m(t) \mathcal A^t_m $ where $d_m(t)= \sum_{\sigma \in \mathfrak S_m} (1/t)^{\ell(\sigma)}$ is a non-zero constant.  With these relations in hand, we can
 relate $[f,g]_m$ to $\Theta_m$.  Indeed, we have
\begin{align}
  \Theta_m(\mathcal A^t_m  f(Y^{-1}) g(x) \mathcal A^t_m  ) & = d_m(t) \, \Theta_m(f(Y^{-1}) g(x) \mathcal A^t_m  ) \nonumber \\
  & = d_m(t) u_{\Lambda_0}^- \bigl( f(Y^{-1}) g(x) \mathcal A^t_m  \cdot E_{\delta_m}(x;q,t) \bigr)  \nonumber \\
  & = d_m(t) c_m(q,t) u_{\Lambda_0}^- \bigl( f(Y^{-1}) g(x) \cdot \Delta_m^t(x) \bigr) \nonumber \\
  &=  d_m(t) c_m(q,t) [f,g]_m \label{symmetry}
\end{align}  
where, in the fist equality, we used the fact that $\mathcal A^t_m $ commutes with $f(Y^{-1})$ and $g(x)$ because they are both symmetric in the first $m$ variables.  We now use  $\Theta_m =  \Theta_m \circ \varphi$ and  $\varphi(\mathcal A^t_m )=\mathcal A^t_m $ to interchange $f$ and $g$:
\begin{equation*}
 \Theta_m\bigl(\mathcal A^t_m  f(Y^{-1}) g(x) \mathcal A^t_m  \bigr) = \Theta_m \circ \varphi \bigl(\mathcal A^t_m  f(Y^{-1}) g(x) \mathcal A^t_m  \bigr)=  \Theta_m\bigl(\mathcal A^t_m  g(Y^{-1}) f(x) \mathcal A^t_m  \bigl) 
\end{equation*}  
The symmetry $[f,g]_m=[g,f]_m$ then immediately follows from \eqref{symmetry}.
\end{proof}

\subsection{Symmetry of the bisymmetric Macdonald polynomials}
We can now extend a well-known result on Macdonald polynomials to the bisymmetric case. But first, we need to give the explicit expressions for $u_{\Lambda_0}^-(\mathcal P_{\Lambda}(x,q,t))$ and
$u_{\Lambda_0}^+(\mathcal P_{\Lambda}(x,q,t))$ that were obtained in \cite{GL}.

For a box $s=(i,j)$ in a partition $\lambda$ (i.e., in row $i$ and column $j$), we introduce the usual arm-lengths and leg-lengths:
\begin{align} 
a_{\lambda}(s)= \lambda_i-j \quad {\rm and} \quad l_{\lambda}(s)= \lambda'_j-i
\end{align}
where we recall that $\lambda'$ stands for the
conjugate of the partition $\lambda$. Let  $\mathcal B(\Lambda)$ denote the set of boxes in the diagram of  $\Lambda$ that do not
appear at the same time in a row containing a circle {and} in a
column containing a circle.
{\begin{equation*}
\Lambda={\tiny{\tableau[scY]{&&&&&\bl\tcercle{}\\&&& &\\&&&\bl\tcercle{}\\ &&&\bl \\&\bl\tcercle{}\\ &\bl \\ & \bl \\ \bl\tcercle{}}}}\quad\Longrightarrow\quad
\mathcal{B}\Lambda={\tiny{\tableau[scY]{\bl&\bl&&\bl&\\&&&&\\ \bl&\bl& \\ & & \\ \bl\\ &\bl\\  &\bl}}}
\end{equation*}}
For $\Lambda$ a superpartition of fermionic degree $m$, let $\mathcal S \Lambda$
be the skew diagrams $\Lambda^{\circledast}/\delta_{m+1}$.
{\begin{equation*}
\Lambda={\tiny{\tableau[scY]{&&&&&\bl\tcercle{}\\&&&&\\&&&\bl\tcercle{}\\ &&\bl \\&\bl\tcercle{}\\  &\bl \\ &\bl  \\ \bl\tcercle{}}}}\quad\Longrightarrow\quad
\mathcal{S}\Lambda={\tiny{\tableau[scY]{\bl&\bl&\bl&\bl&&\\\bl&\bl&\bl&&\\ \bl&\bl&&\\\bl &&\bl \\&\\ &\bl \\ &\bl \\ & \bl }}}  
\end{equation*}}
Finally, for a partition $\lambda$, let $n(\lambda)=\sum_{i} (i-1)\lambda_i$.
In the case of a skew partition $\lambda/\mu$,  $n(\lambda/\mu)$ stands for $n(\lambda)-n(\mu)$.

The following theorem was proved in  \cite{GL} for Macdonald polynomials in superspace (in fact, only  \eqref{ev-mac}
was proved therein.  But
using \eqref{eqtneg}, one can immediately deduce
\eqref{ev+mac}). From Remark~\ref{remark1}, it also applies to bisymmetric Macdonald polynomials.

\begin{theorem} 
Let $\La$ be of fermionic degree $m$.
Then the evaluation formulas for the bisymmetric Macdonald polynomials  read
\begin{equation}
\label{ev-mac}
u_{\Lambda_0}^- \bigl(\mathcal P_\La\bigr)= \frac{t^{n(\mathcal S \Lambda)+n( (\Lambda')^a/\delta_m)}}
   {q^{(m-1) |\Lambda^a/\delta_m| - n(\Lambda^a/\delta_m)}}
   \frac{\prod_{(i,j)\in \mathcal{S}\Lambda}(1-q^{j-1}t^{N-(i-1)})}{\prod_{s\in \mathcal{B}\Lambda} (1-q^{a_{\Lambda^{\circledast}}(s)}t^{l_{\Lambda^*}(s)+1})} 
\end{equation}
and
\begin{equation} \label{ev+mac}
  u_{\Lambda_0}^+
  \bigl(\mathcal P_\La\bigr)= \frac{q^{m |\Lambda^a/\delta_m| - n(\Lambda^a/\delta_m)}}
   {t^{n(\mathcal S \Lambda)+n( (\Lambda')^a/\delta_m)}}
   \frac{\prod_{(i,j)\in \mathcal{S}\Lambda}(1-q^{1-j}t^{i-(N+1)})}{\prod_{s\in \mathcal{B}\Lambda} (1-q^{-a_{\Lambda^{\circledast}}(s)}t^{-l_{\Lambda^*}(s)-1})} 
\end{equation}
\end{theorem}
We can now state the two symmetries satisfied by the bisymmetric Macdonald polynomials.
\begin{theorem}\label{SimMac} Let $\tilde{\mathcal P}_{\Lambda}^-(x,q,t)$ and  $\tilde{\mathcal P}_{\Lambda}^+(x,q,t)$ be the two normalizations of the 
bisymmetric Macdonald polynomials:
  $$
\tilde{\mathcal P}_{\Lambda}^-(x,q,t) = \dfrac{\mathcal P_{\Lambda}(x;q,t)}{u_{\Lambda_0}^-(\mathcal P_{\Lambda}(x,q,t))} \qquad {\rm and} \qquad
\tilde{\mathcal P}_{\Lambda}^+(x,q,t) = \dfrac{\mathcal P_{\Lambda}(x;q,t)}{u_{\Lambda_0}^+(\mathcal P_{\Lambda}(x,q,t))} 
$$
where we recall that $\Lambda_0$ was defined in Remark~\ref{remark1}.
Then, the following two symmetries hold:
$$
u_\Omega^-(\tilde{\mathcal P}_{\Lambda}^-)=u_\Lambda^-(\tilde{\mathcal P}_{\Omega}^-)
\qquad {\rm and} \qquad u_\Omega^+(\tilde{\mathcal P}_{\Lambda}^+)=u_\Lambda^+(\tilde{\mathcal P}_{\Omega}^+)
$$
\end{theorem} 
\begin{proof} We first prove the symmetry involving the negative evaluation.
  From the definition of the pairing $[\cdot,\cdot]$ and from Lemma~\ref{AutVal}, we get that
\begin{equation}
    \begin{array}{ll}
    \left[{\mathcal P}_{\Lambda}(x,q,t), {\mathcal P}_{\Omega}(x,q,t) \right]_m &= u_{\Lambda_0}^-\bigl({\mathcal P}_{\Lambda}(Y^{-1}_i) \Delta_m^t (x) {\mathcal P}_{\Omega}(x,q,t)\bigr) \\
    &= u_\Omega^-({\mathcal P}_{\Lambda}(x,q,t)) u_{\Lambda_0}^-\bigl( \Delta_m^t(x) {\mathcal P}_{\Omega}(x,q,t)\bigr)
\end{array}
\end{equation}
Using Lemma~\ref{SimGen}, we then have
$$
u_\Omega^-({\mathcal P}_{\Lambda}(x,q,t)) u_{\Lambda_0}^-\bigl( \Delta_m^t(x) {\mathcal P}_{\Omega}(x,q,t)\bigr)=
 u_\Lambda^-({\mathcal P}_{\Omega}(x,q,t)) u_{\Lambda_0}^-\bigl( \Delta_m^t(x) {\mathcal P}_{\Lambda}(x,q,t)\bigr)
$$
and the first symmetry follows immediately.

We will now deduce the symmetry involving the positive evaluation from the negative one. As we will see, it essentially follows from Equation (4.6) 
in \cite{ABDLM} which, when rewritten in our language, says that
\begin{equation*}
\Delta_m(qx_1,\dots,q x_m) \mathcal P_\Lambda(qx_1,\dots,qx_m,x_{m+1},\dots,x_N;q,t) =q^{|\Lambda^a|}\Delta_m(x) \mathcal P_\Lambda(x;1/q,1/t)
\end{equation*}    
Simplifying the previous equation as
\begin{equation} 
q^{m(m-1)/2-|\Lambda^a|} \mathcal P_\Lambda(qx_1,\dots,qx_m,x_{m+1},\dots,x_N;q,t) =\mathcal P_\Lambda(x;1/q,1/t)
\end{equation} 
we deduce that
\begin{align}
\label{eqtneg}
  & \bigl[u_\Omega^- (\mathcal P_\Lambda(x;q,t))\bigr]_{\substack{ (q,t) \mapsto (1/q, 1/t) }} \nonumber \\
& \quad  =
  \mathcal P_\Lambda( q^{\Omega_1}t^{1-w(1)},\dots,q^{\Omega_m}t^{1-w(m)}, q^{\Omega_{m+1}}t^{1-w(m+1)},\dots,q^{\Omega_N}t^{1-w(N)};1/q,1/t) \nonumber\\
   & \quad  =
  q^{m(m-1)/2-|\Lambda^a|}  \mathcal P_\Lambda( q^{\Omega_1+1}t^{1-w(1)},\dots,q^{\Omega_m+1}t^{1-w(m)}, q^{\Omega_{m+1}}t^{1-w(m+1)},\dots,q^{\Omega_N}t^{1-w(N)};q,t) \nonumber\\
& \quad  = q^{m(m-1)/2-|\Lambda^a|}  u_{\Omega}^+(\mathcal P_{\Lambda})
\end{align}
It is then immediate that
\begin{equation*}
    \begin{array}{ll}
    \bigl[u_\Omega^- (\mathcal P_\Lambda(x;q,t))\bigr]_{\substack{ (q,t) \mapsto (1/q, 1/t) }}     
    &= \left[ \dfrac{ u_{\Omega}^- (\mathcal P_{\Lambda}(x;q,t))}{u_{\Lambda_0}^-(\mathcal P_{\Lambda}(x,q,t))} \right]_{\substack{ (q,t) \mapsto (1/q, 1/t) }}   \\
    &= \dfrac{ u_{\Omega}^+ (\mathcal P_{\Lambda}(x;q,t))}{u_{\Lambda_0}^+(\mathcal P_{\Lambda}(x,q,t))} \\
    &= u_\Omega^+(\tilde{\mathcal P}_{\Lambda}^+)
    \end{array}
\end{equation*}
from which we conclude that the second symmetry also holds.
\end{proof}

\section{Double affine Hecke algebra relations}
In this section, we establish a few results involving the Hecke algebra and the 
Double affine Hecke algebra. They will be needed in the next section.  We start with a generalization of a known result in symmetric function theory.
\begin{lemma}\label{RelU+} Let $J \subseteq [N]$ and $L = [N] \setminus J$.
We then have
$$\displaystyle \sum_{\substack{\sigma([N-r+1,N]) = J \\ \sigma \in  \mathfrak S_{N}}} K_\sigma
 \left(\prod_{1 \leq i<j \leq N} \dfrac{x_i-t x_j}{x_i-x_j} \right)
 = a_{r,N}(t) A_{J \times L}(x,x)$$
where $r=|J|$ and
$$
a_{r,N}(t)= [r]_t! [N-r]_t!
$$
\end{lemma}

\begin{proof}  For convenience, we will let
  $$
\bar A_I(x) =  \prod_{\substack{ i,j \in I \\ i < j}} \left( \frac{x_i-tx_j}{x_i-x_j} \right)
  $$
We first prove the special case when $J=[r]$ and $L=[r+1,N]$.
Let $\gamma$ be the permutation $[r+1,\dots,N,1,\dots,r]$ (in one-line notation). In this case, we have
\begin{align*}
&  \displaystyle \sum_{\substack{\sigma([N-r+1,N]) = [r] \\ \sigma \in  \mathfrak S_{N}}} K_\sigma
 \bar A_N(x) \\
 & \qquad \qquad = \sum_{w \in  \mathfrak S_{r}} \sum_{w' \in  \mathfrak S_{r+1,N} }K_w K_{w'} K_\gamma
\bar  A_{N-r}(x) \bar A_{[N-r+1,N]}(x) A_{[N-r+1,N] \times [N-r]}(x,x) \\
  & \qquad \qquad = A_{[r] \times [r+1,N]}(x,x) \left( \sum_{w \in  \mathfrak S_{r}} K_w \bar A_{r}(x)
\right) \left(\sum_{w' \in  \mathfrak S_{r+1,N} } K_{w'} \bar A_{[r+1,N]}(x) \right)
\end{align*}
since $w$ and $w'$ leave $A_{[r] \times [r+1,N]}$ invariant.  Using \cite{Mac}
\begin{equation} \label{eqHecke}
\mathcal S_N^t \cdot 1= \sum_{ \sigma \in  \mathfrak S_{N}} K_\sigma  \bar A_{N}(x) = [N]_t!
\end{equation}
the formula is seen to hold in that case.

As for the general case, let $\delta$ be any permutation such that
$\delta([r])=J$ (and thus also such that $\delta([r+1,\dots,N])=L$). Applied on both sides of the special case that we just showed, we get
$$
\displaystyle \sum_{\substack{\sigma([N-r+1,N]) = [r] \\ \sigma \in  \mathfrak S_{N}}} K_\delta K_\sigma \left(\prod_{1 \leq i<j \leq N} \dfrac{x_i-t x_j}{x_i-x_j} \right) =  a_{r,N}(t) K_\delta  A_{[r] \times [r+1,N]}(x,x)=  a_{r,N}(t)  A_{J \times L}(x,x)
$$
which amounts to
$$
\displaystyle \sum_{\substack{\delta \sigma([N-r+1,N]) = J \\ \delta \sigma \in  \mathfrak S_{N}}} K_{\delta \sigma}  \left(\prod_{1 \leq i<j \leq N} \dfrac{x_i-t x_j}{x_i-x_j} \right)
=  a_{r,N}(t)  A_{J \times L}(x,x)
$$
The lemma then follows immediately.
\end{proof}
We now show that the product $Y_{N-r+1} \cdots Y_N$ of Cherednik operators can be simplified quite significantly in certain cases.
\begin{lemma} \label{rel Y}  Let $r \leq N-m$.
For any bisymmetric function $f(x)$, we have that    
        \begin{equation*}
               Y_{N-r+1} \cdots Y_N  f(x) = t^{(2m+r+1-2N)r/2}  \omega^{r} (\bar{T}_r \cdots \bar{T}_{m+r-1}) \cdots (\bar{T}_1 \cdots \bar{T}_{m})  f(x)
          \end{equation*}
\end{lemma}
\begin{proof}  We first show that
\begin{equation} \label{firsteq}
  Y_{N-r+1} \cdots Y_N =
  t^{-r(r-1)/2}(\omega \bar{T}_1 \cdots \bar{T}_{N-r})^r
\end{equation}
The result obviously holds by definition when $r=1$.  Assuming that it holds for $r-1$, we have that
\begin{align*}
  Y_{N-r+1} \cdots Y_N & = Y_N \cdots Y_{N-r+1}  \\
  &= 
  t^{-(r-1)(r-2)/2}(\omega \bar{T}_1 \cdots \bar{T}_{N-r+1})^{r-1} \bigl( t^{-r+1}T_{N-r+1} \cdots T_{N-1} \omega \bar T_{1} \cdots \bar T_{N-r} \bigr)
\end{align*}
Making use of the relation
$\bar{T}_{i-1} \omega = \omega \bar{T}_{i}$, we can move the term $\bar T_{N-r+1}$ of every product to the right to get
\begin{align*}
  Y_{N-r+1} \cdots Y_N & =  t^{-r(r-1)/2}(\omega \bar{T}_1 \cdots \bar{T}_{N-r})^{r-1} \bar T_{N-1} \cdots \bar T_{N-r+1}  T_{N-r+1}  \cdots T_{N-1} \omega \bar T_{1} \cdots \bar T_{N-r}  \\
& = t^{-r(r-1)/2}(\omega \bar{T}_1 \cdots \bar{T}_{N-r})^r  
\end{align*}
which proves \eqref{firsteq} by induction.

Using 
$\bar{T}_{i-1} \omega = \omega \bar{T}_{i}$ again and again, we then get from \eqref{firsteq} that
$$
 Y_{N-r+1} \cdots Y_N = t^{-r(r-1)/2} \omega^r ( \bar{T}_r \cdots \bar{T}_{N-1}) \cdots (\bar{T}_1 \cdots \bar{T}_{N-r}) 
$$
If $f(x)$ is a bisymmetric function, the rightmost $N-r-m$ terms in every product in the previous equation can be pushed to the right and made to act as $1/t$ on $f(x)$.  This yields,
\begin{equation*}
          Y_{N-r+1} \cdots Y_N  f(x) = t^{-r(N-r-m) - r(r-1)/2}   \, \omega^{r} (\bar{T}_r \cdots \bar{T}_{m+r-1}) \cdots (\bar{T}_1 \cdots \bar{T}_{m})  f(x)
\end{equation*}
which proves the lemma.
\end{proof}
The next result shows that $e_r(Y_{1}, \ldots, Y_N)$ can be recovered from 
$\mathcal S^t_N$ acting on $Y_{N-r+1} \cdots Y_N$. 
\begin{lemma}  \label{lemmaeY}
  For $r \leq N$, we have that if $f(x)$ is a symmetric function  then
\begin{equation*}
e_r(Y_{1}, \ldots, Y_N) f(x)= \dfrac{1}{[N-r]_t! [r]_t!} \mathcal S^t_N  Y_{N-r+1} \cdots Y_N   f(x)     
\end{equation*}
\end{lemma}
\begin{proof}
  First, if $w \in  \mathfrak S_{r}$ and $\sigma \in  \mathfrak S_{r+1,N}$ then $(T_w T_\sigma) Y_{N-r+1} \cdots Y_N = Y_{N-r+1} \cdots Y_N (T_w T_{\sigma})$
 by \eqref{TYi}. This yields
    $$
T_w T_\sigma Y_{N-r+1} \cdots Y_N f(x)= t^{\ell(w)+\ell(\sigma)} Y_{N-r+1} \cdots Y_N f(x)
$$
given that $f(x)$ is symmetric.
Hence,
summing over all the elements of $ \mathfrak S_r \times  \mathfrak S_{r+1,N}$
in $\mathcal S^t_N=\sum_{\sigma \in \mathfrak S_N} T_\sigma$
gives a factor of
$[N-r]_t! [r]_t!$ from \eqref{eqHecke}.  We thus have left to prove that
$$
e_r(Y_{1}, \ldots, Y_N) f(x)= \sum_{[\sigma^*] \in  \mathfrak S_N/ ( \mathfrak S_r \times  \mathfrak S_{r+1,N})} T_{\sigma^*}  Y_{N-r+1} \cdots Y_N   f(x) 
$$
where the sum is over all left-coset representatives $\sigma^*$ of minimal length. Such minimal length representatives are of the form (in one-line notation) $\sigma^*=[i_1,\dots,i_{N-r},i_{N-r+1},\dots,i_{N}]$ with $i_1 < i_2 < \cdots < i_{N-r}$ and $i_{N-r+1} < i_{N-r+2} < \cdots < i_N$. A reduced decomposition of $\sigma^*$ is then given by \begin{equation} \label{reddecomp}
 (s_{i_N} \cdots s_{N-1}) \cdots (s_{i_{N-r+1}} s_{i_{N-r+1}+1}\dots s_{N-r})
\end{equation}
We will now see that the factor $T_{i_{N-r+1}} T_{i_{N-r+1}+1}\dots T_{N-r}$ of $T_{\sigma^*}$ changes $Y_{N-r+1}$ into $Y_{i_{N-r+1}}$ and leaves the rest of the terms invariant. First, we use the relation $ T_i Y_{i+1}=t Y_i \bar T_i$ to obtain
$$
 T_{N-r} Y_{N-r+1} Y_{N-r+2} \cdots Y_N f(x)= t Y_{N-r}  \bar T_{N-r} Y_{N-r+2} \cdots Y_N f(x)=  Y_{N-r}   Y_{N-r+2} \cdots Y_N f(x)
 $$
 Proceeding in this way again and again, we then get that
 $$
 T_{i_{N-r+1}} T_{i_{N-r+1}+1}\dots T_{N-r}  Y_{N-r+1} Y_{N-r+2} \cdots Y_N f(x)
 = Y_{i_{N-r+1}}  Y_{N-r+2} \cdots Y_N f(x)
 $$
 as wanted.  By assumption, all of the remaining indices of the $s_j$'s in \eqref{reddecomp} are larger than $i_{N-r+1}$.  Hence $Y_{i_{N-r+1}}$ will not be affected by the remaining terms in $T_{\sigma^*}$. Following as we just did, it is then immediate that
 $$
 T_{\sigma^*}  Y_{N-r+1} \cdots Y_N   f(x) = Y_{i_{N-r+1}} \cdots Y_{i_N}  f(x)
 $$
Finally, summing over all $\sigma^*$, the lemma is then seen to hold.
 \end{proof}

\section{The action of  $e_{r}(Y_{m+1}, \ldots ,Y_{N})$ on bisymmetric functions}
\label{sec4}
In this section, we will obtain the  explicit action of the operator  $e_{r}(Y_{m+1}, \ldots ,Y_{N})$  on a bisymmetric function.  This will then be used in the next section to deduce the Pieri rules $e_{r}(x_{m+1}, \ldots ,x_{N})$ for the bisymmetric Macdonald polynomials.  
We first need to set some notation.

Let $[N]=\{1,\dots ,N\}$. For $I$ a  subset of $[N]$, recall that
\begin{equation}
\Delta_I(x)= \prod_{\substack{ i,j \in I \\ i < j}} (x_i-x_j)\, , \qquad
\Delta_I^t(x)= \prod_{\substack{ i,j \in I \\ i < j}} (tx_i-x_j) \qquad {\rm and}
\qquad A_I(x)=  \prod_{\substack{ i,j \in I \\ i < j}} \left( \frac{tx_i-x_j}{x_i-x_j} \right)
\end{equation}
For simplicity, when $I=[m]= \{1,\dots,m \}$, we will use the notation
$\Delta_m(x)$, $\Delta_m^t(x)$ or $A_m(x)$  instead of $\Delta_{[m]}(x)$, $\Delta_{[m]}^t(x)$ or $A_{[m]}(x)$.

We will also need a similar notation for subsets of $[N] \times [N]$.
If $\Aset \subseteq [N] \times [N]$, we let 
$$R_{\Aset}(x,y) = \prod_{(i,j) \in \Aset } (1-x_i y_j)\, , \quad
\varDelta_{\Aset }(x,y) = \prod_{(i,j) \in \Aset } (x_i-y_j)\, , \quad {\rm and} \quad 
A_{\Aset}(x,y) = \prod_{(i,j) \in \Aset  } \left(\frac{tx_i - y_j}{x_i-y_j} \right)$$
With this notation in hand, we define
$$\mathcal{F}_{m}(x,y) = \dfrac{\Delta^t_{m}(x)}{R_{[m] \times [m]}(x,y)} $$
and
$$\mathcal{NF}_{m}(x,y) = \dfrac{R_{\Bset}(x,ty)}{R_{\Bset'}(x,y)} $$
where $\Bset$ is the set of integer points in the triangle with vertices $(1,1),(1,m-1)$ and $(m-1,1)$,
while $\Bset'$ is the set of integer points in the triangle with vertices
$(1,1),(1,m)$ and $(m,1)$.

\begin{example} The product $\mathcal{NF}_{3}(x,y)$ can be seen at the quotient of the factors stemming from the two following regions
\begin{center}
\includegraphics[scale=0.3]{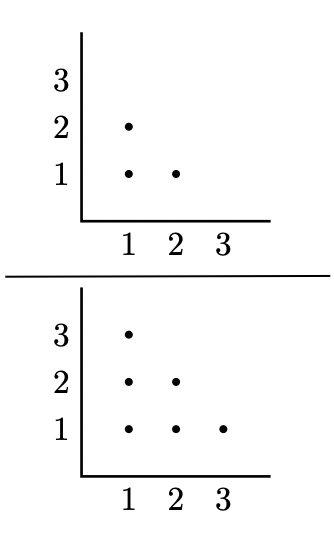}
\end{center}
Hence
$\mathcal{NF}_{3}(x,y)$ is equal to
$$\frac{(1-tx_1 y_1)(1-tx_1y_2)(1-tx_2y_1)}{(1-x_1 y_1)(1-x_1y_2)(1-x_1 y_3)(1-x_2 y_1)(1-x_2 y_2)(1-x_3 y_1)}
$$
\end{example}

Let $\mathcal A^{(y)}_m$ stand for the antisymmetrizer
$\mathcal A_m$ defined in  \eqref{antisym} but acting on the 
$y$ variables instead of the $x$ variables.
The next lemma was proven in \cite{BDLM2} in another form.
\begin{lemma}\label{RelR}  We have
$$\mathcal{A}^{(y)}_m R_{\Cset}(x,ty) R_{\Cset'}(x,y) = (-1)^{\binom{m}{2}} \Delta_m^t(x)\Delta_m(y)$$
where $\Cset$ is the set of integer points in the triangle with vertices $(1,1),(1,m-1)$ and $(m-1,1)$ while $\Cset'$ is the set of integer points in in the triangle with vertices $(2,m),(m,m)$ and $(m,2)$.
\end{lemma}

\begin{proof} Equation (129) in \cite{BDLM2} rewritten in our language (and with the $t$-power corrected) says that
$$
\frac{1}{\Delta_m(y)}  \mathcal{A}^{(y)}_m \left( \prod_{i+j \leq m} (1-tx_i y_j) \prod_{\substack{i+j > m+1\\ i,j \leq m}}  (1-x_i y_j) \right) =  (-1)^{\binom{m}{2}} \Delta_m^t(x)
$$
The lemma then immediately follows.
\end{proof}

\begin{corollary} \label{RelB} The following identity holds:
 $$\mathcal{A}^{(y)}_m \mathcal{NF}_m(x,y) = (-1)^{\binom{m}{2}}\Delta_m(y)\mathcal{F}_m(x) $$
\end{corollary}

\begin{proof} 
The identity can be deduced from Lemma~\ref{RelR} after completing the square in the triangle $\mathcal B'$ corresponding to the denominator of $\mathcal{NF}_{m}(x,y)$.\end{proof}

We now establish a few elementary relations that will be needed later on.
\begin{lemma} \label{RelTR}  For all $i \in \{ 1,\dots,N-1\}$ and all $j \in \{ 1,\dots,N\}$, 
we have
\begin{enumerate}
    \item $\bar T_{i} \dfrac{1}{(1-x_i y_j)} = \dfrac{(1-tx_i y_j)}{t(1-x_i y_j)(1-x_{i+1}y_j)}$
\item $T_{i} \dfrac{1}{(1-x_{i+1} y_j)} = \dfrac{t (1-t^{-1 }x_{i+1} y_j)}{(1-x_{i+1}y_j) (1-x_i y_j)}
    $ 
\end{enumerate}
\end{lemma}
\begin{proof}
  We only prove the first relation  as the second one can be proven in the same fashion.  Since $(1-x_i y_j)(1-x_{i+1}y_j)$ is symmetric in $x_i$ and $x_{i+1}$, it commutes with $\bar T_i$. The first relation is thus equivalent to $\bar T_i(1-x_{i+1} y_j)= t^{-1}(1-tx_i y_j)$, which can easily be verified. 
\end{proof}

The following lemma concerns the function
\begin{equation}
  \mathcal{NF}_{m}^{k}(x,y) =
  \dfrac{R_{\Bset_k}(x,ty)}{R_{\Bset'_k}(x,y)} \end{equation}
  where $\Bset_k$ is the set of integer points in the trapezoid with vertices $(1,1),(1,m),(k,m)$ and $(m+k-1,1)$ while  $\Bset_k'$ is the set of integer points in the trapezoid with vertices $(1,1),(1,m),(k+1,m)$ and $(m+k,1)$. That is,  $\Bset_k'$ is the union of  $\Bset_k$ with the line segment from $(k+1,m)$ to $(m+k,1)$ of slope $-1$.  Note that  $\mathcal{NF}_m^0(x,y) = \mathcal{NF}_m(x,y)$.

\begin{lemma} \label{TnegNF} 
For $k \geq 1$, we have
$$\bar T_{k} \cdots \bar T_{k+m-1} \mathcal{NF}_{m}^{k-1}(x,y) = t^{-m} \mathcal{NF}_{m}^{k}(x,y)$$
Consequently, for $r\geq 1$, 
\begin{equation} 
(\bar T_{r} \cdots \bar T_{r+m-1}) \cdots (\bar T_{1} \cdots \bar T_{m}) \mathcal{NF}_{m}(x,y) =t^{-rm} \mathcal{NF}_{m}^{r}(x,y)     
\end{equation}
\end{lemma}

\begin{proof}  
The lemma follows from
Lemma~\ref{RelTR}. Diagrammatically,  acting with $\bar T_{m}$ on $\mathcal{NF}_{m}$ amounts to adding a dot to the triangles associated to the numerator and the denominator in the diagram of $\mathcal{NF}_{m}$. For instance, if $m=3$ the diagram associated to $\bar T_{3} \mathcal{NF}_{3}$ is:
\begin{center}
\includegraphics[scale=0.3]{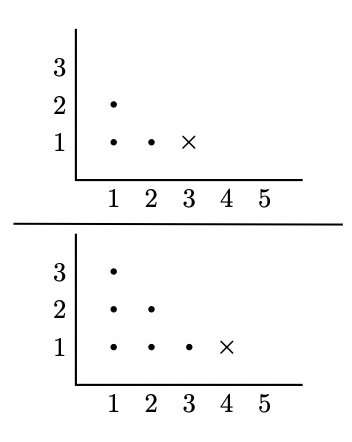}
\end{center}
The product $\bar T_{1} \cdots \bar T_{m}$ adds a diagonal above the triangles associated to $\mathcal{NF}_{m}$. For example,   $\bar T_{1} \bar T_{2} \bar T_{3} \mathcal{NF}_{3}$ is
\begin{center}
\includegraphics[scale=0.3]{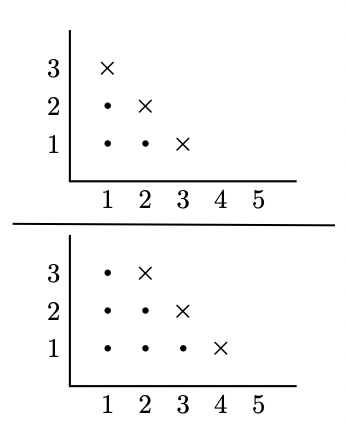}
\end{center}
Finally, as $r$ increases, extra diagonals are added. We get for instance, in the case $r=2$ and $m=3$, that acting with  $(\bar T_{2} \bar T_{3} \bar T_{4})( \bar T_{1} \bar T_{2} \bar T_{3})$ on  $\mathcal{NF}_{3}$ adds the following two diagonals:
\begin{center}
\includegraphics[scale=0.3]{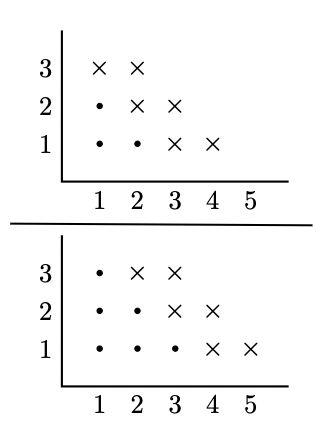}
\end{center}
\end{proof}
For the next lemma, we need to introduce some notation. Given a permutation $\sigma \in \mathfrak S_N$, we let
$$\AA_{\sigma} = \{i \in [1,m] \ |\ \sigma^{-1}(i) \in [1,m]\} \quad {\rm and }\quad
\BB_{\sigma} = \{i \in [m+1,N] \ |\ \sigma^{-1}(i) \in  [m+1,N]\}$$
Their respective complements are
$$\AA_{\sigma}^c = [1,m] \setminus \AA_{\sigma}, \quad {\rm and} \quad  \BB_{\sigma}^c = [m+1,N] \setminus \BB_{\sigma}$$
\begin{remark} \label{remarkinv}
  It is important to realize that   if  $w \in  \mathfrak S_m \times  \mathfrak S_{m+1,N}$ then $\AA_{\sigma}=\AA_{\sigma w}$ (and similarly for $\AA_{\sigma}^c$, $\BB_{\sigma}$ and $\BB_{\sigma}^c$).
\end{remark}
 For $A\times B \subseteq [N] \times [N]$, let  also
$(A \times B)_<= \{(i,j) \in A \times B \, |\,  i<j\}$. 
Finally, we define
$$
\widetilde \varDelta_{\Aset }(x,y) = \prod_{\substack{(i,j) \in \Aset \\ i \neq j } } (x_i-y_j) \quad {\rm and} \quad 
\widetilde A_{\Aset}(x,y) = \prod_{\substack{(i,j) \in \Aset \\ i \neq j } } \left(\frac{tx_i - y_j}{x_i-y_j} \right)$$
\begin{lemma}\label{RelDet} For  $\sigma \in \mathfrak S_{N}$, let
\begin{equation} \label{varphi}
    \varPhi(\sigma) = (-1)^{D_\sigma}\left( \prod_{j \in \BB_{\sigma}^c}(t-1)x_j \right)\dfrac{\Delta_{\AA_{\sigma}^c}(x)\Delta_{\BB_{\sigma}^c}(x)}{\Delta_{\AA_{\sigma}^c \times \BB_{\sigma}^c}(x,x)}
\end{equation}
where
$$D_{\sigma} = \dfrac{s(s-1)}{2}+\#(\AA_{\sigma}^c\times \AA_{\sigma})_<  +Z_{\sigma}$$
with $s= \# \AA_{\sigma^c}$ and $Z_\sigma=\#\{(i,j) \in [m] \times [m] \, |\, \sigma(i) > \sigma(j) \}$.
We then have the following equality:
\begin{equation} \label{eqPHI}
  \varPhi(\sigma) =\left(\prod_{j \in \BB_{\sigma}^c}(t-1)x_j \right) \cdot \dfrac{\widetilde{\varDelta}_{\BB_{\sigma}^c \times \sigma([m])}(x,x)}{\varDelta_{\AA_{\sigma}^c \times \sigma([m])}(x,x)}  \dfrac{\Delta_{m}(x)}{K_\sigma( \Delta_{m}(x))}
\end{equation}  
Moreover, if  $w \in  \mathfrak S_m \times  \mathfrak S_{m+1,N}$ then
\begin{equation} \label{eqinv}
\varPhi(\sigma w) = (-1)^{Z_{w}}\varPhi(\sigma)
\end{equation}
\end{lemma}
\begin{proof} From Remark~\ref{remarkinv}, \eqref{eqinv}  is easily seen to hold given that
 $K_w(\Delta_m(x)) = (-1)^{Z_w} \Delta_{\{w(1),\ldots,w(m)\}}(x)$.
 We now prove \eqref{eqPHI}.
After simplifying the terms $\prod_{j \in \BB_{\sigma}^c}(t-1)x_j$
in $\varPhi(\sigma)$ and in the r.h.s. of \eqref{eqPHI}, we have left to prove that
\begin{equation} \label{topro}
  (-1)^{D_\sigma}\dfrac{\Delta_{\AA_{\sigma}^c}(x)\Delta_{\BB_{\sigma}^c}(x)}{\Delta_{\AA_{\sigma}^c \times \BB_{\sigma}^c}(x,x)} =
 \dfrac{\widetilde{\varDelta}_{\BB_{\sigma}^c \times \sigma([m])}(x,x)}{\varDelta_{\AA_{\sigma}^c \times \sigma([m])}(x,x)}  \dfrac{\Delta_{m}(x)}{K_\sigma( \Delta_{m}(x))}
\end{equation}
As all the remaining products are of the form
  $(x_i-x_j)$, it will prove convenient to simply work with sets, taking special care of the signs that may appear. On the r.h.s. of \eqref{topro}, we have in the numerator
  $$\bigl( \BB_{\sigma}^c \times \sigma([m]) \bigr) \cup ([m] \times [m])_<$$
  We  observe that $\sigma([m])=\AA_\sigma \cup \BB_\sigma^c$ since $\sigma^{-1}(i) \in [m] \iff i \in \sigma([m])$.  Hence, the numerator on the r.h.s. of \eqref{topro} is equal to
\begin{equation}  \label{eqnum}
 (\BB_{\sigma}^c \times \AA_{\sigma}) \cup (\BB_{\sigma}^c \times \BB_{\sigma}^c)_< \cup (\BB_{\sigma}^c \times \BB_{\sigma}^c)_>   \cup  (\AA_{\sigma} \times \AA_{\sigma})_<  \cup  (\AA_{\sigma} \times \AA_{\sigma}^c)_<  \cup ( \AA_{\sigma}^c \times \AA_{\sigma})_<  \cup  (\AA_{\sigma}^c \times \AA_{\sigma}^c)_<
\end{equation}
Now, the denominator on the r.h.s. is equal, up to a sign $(-1)^{Z_\sigma}$, to
$$ \bigl(\AA_{\sigma}^c \times \sigma([m]) \bigr) \cup  \bigl(\sigma([m]) \times \sigma([m]) \bigr)_< $$
which is in turn equivalent to
\begin{equation} \label{eqdenom}
  (\AA_{\sigma}^c \times \AA_{\sigma}) \cup (\AA_{\sigma}^c \times \BB_{\sigma}^c) \cup (\AA_{\sigma} \times \AA_{\sigma})_<  \cup (\AA_{\sigma} \times \BB_{\sigma}^c)_<  \cup (\BB_{\sigma}^c \times \AA_{\sigma})_<  \cup (\BB_{\sigma}^c \times \BB_{\sigma}^c)_< \end{equation}
It is immediate that $A \times B= (A\times B)_> \cup (A \times B)_<$ if $A$ and $B$ are disjoint. Moreover,  $(A\times B)_>=(B \times A)_<$ (which accounts for an extra sign $(-1)^{\#{(B \times A)_<}}$). Hence, 
comparing \eqref{eqnum} and \eqref{eqdenom}, we have that 
$(\AA_{\sigma}^c \times \AA_{\sigma}^c)_< \cup  (\BB_{\sigma}^c \times \BB_{\sigma}^c)_<$ is left on the numerator while $(\AA_{\sigma}^c \times \BB_{\sigma}^c)_<=\AA_{\sigma}^c \times \BB_{\sigma}^c$ is left on the denominator, with the extra sign being
$$(-1)^{ \# (\BB_{\sigma}^c\times \BB_{\sigma}^c)_<+\#(\AA_{\sigma}\times \BB_{\sigma}^c)_< +\#(\AA_{\sigma}\times \AA_{\sigma}^c)_<}$$
Taking into account the sign  $(-1)^{Z_\sigma}$ obtained earlier, we obtain
$$
D_\sigma=  \# (\BB_{\sigma}^c\times \BB_{\sigma}^c)_<+\#(\AA_{\sigma}\times \BB_{\sigma}^c)_< +\#(\AA_{\sigma}\times \AA_{\sigma}^c)_< + Z_\sigma
$$
Observer that $\#(\BB_{\sigma}^c\times \BB_{\sigma}^c)_< = \dfrac{s(s-1)}{2}$ since
$s=\# \AA_{\sigma}^c=\# \BB_{\sigma}^c$.  The elements of 
 $\AA_{\sigma}$ being all smaller than those of $\BB_{\sigma}^c$, we get
$$\#(\AA_{\sigma} \times \BB_{\sigma}^c)_< = \#\AA_{\sigma} \cdot \# \BB_{\sigma}^c = (m-s)s$$
Finally, given that the sets $\AA_{\sigma}$ are disjoint $ \AA_{\sigma}^c$, we have
$$
\#(\AA_{\sigma} \times \AA_{\sigma}^c)_< + \#(\AA_{\sigma}^c \times \AA_{\sigma})_< = 
 \#\AA_{\sigma}^c \cdot \# \AA_{\sigma} = s(m-s)$$
 We thus have as wanted that
$$(-1)^{D_{\sigma}} = (-1)^{s(s-1)/2+ \#(\AA_{\sigma}^c \times \AA_{\sigma})_<+Z_{\sigma}}$$
\end{proof}
Before proving the main result of this section, we obtain a criteria to show the equivalence of two operators.
\begin{lemma} \label{lemmasame}
  Let $\mathcal O$ and $\mathcal O'$ be any operators acting on bisymmetric functions. If for all symmetric functions $g(x)$ we have
  $$
  \mathcal O \left( \frac{g(x)}{R_{[m] \times [m]}(x,y)}\right) =  \mathcal O' \left( \frac{g(x)}{R_{[m] \times [m]}(x,y)}  \right) 
   $$
    then
  $$
\mathcal O f(x) = \mathcal O' f(x) 
  $$
for all bisymmetric functions $f(x)$.
\end{lemma}  
\begin{proof} A basis of the space of bisymmetric functions is provided by products of Schur functions $\{ s_\lambda (x_1,\dots,x_m) s_\mu (x_1,\dots,x_N) \, | \, \}_{\lambda,\mu}$ where
  $\lambda$ and $\mu$ are partitions of length not larger than $m$ and $N$ respectively. It is well-known that \cite{Mac}
  $$
\frac{1}{R_{[m] \times [m]}(x,y)} = \sum_{\lambda \, ; \, \ell(\lambda) \leq m} s_\lambda(x_1,\dots,x_m)  s_\lambda(y_1,\dots,y_m) 
$$
Hence, by hypothesis,
$$
0= (\mathcal O -\mathcal O') \left( \frac{s_\mu(x)}{R_{[m] \times [m]}(x,y)}\right)=  \sum_{\lambda \, ; \, \ell(\lambda) \leq m} (\mathcal O -\mathcal O')\bigl(s_\lambda(x_1,\dots,x_m) s_{\mu}(x) \bigr) s_\lambda(y_1,\dots,y_m) 
$$
Taking the coefficient of $s_\lambda(y_1,\dots,y_m) $ in the expansion tells us that the action of $\mathcal O-\mathcal O'$ on the basis element $s_\lambda(x_1,\dots,x_m) s_{\mu}(x)$ is null. We thus conclude that
 $\mathcal O$ and $\mathcal O'$ have the same action on the basis element $s_\lambda(x_1,\dots,x_m) s_{\mu}(x)$, and thus on any bisymmetric function. 
\end{proof}  
\begin{proposition} \label{Op+}
Let $f(x)$ be any bisymmetric function.  Then
\begin{equation*}
  e_{r}(Y_{m+1},\ldots,Y_{N}) \Delta_m^t(x) f(x) =  \sum_{\substack{ J \subset [m+1,N] \\ |J| = r }} \, \, 
  \sum_{\substack{ [\sigma] \in  \mathfrak S_N/ ( \mathfrak S_m \times  \mathfrak S_{m+1,N}) \\ \sigma([m]) \cap L= \emptyset }} C_{J,\sigma}(x) \tau_{J} K_\sigma f (x)
\end{equation*}
where the coefficient $C_{J,\sigma}(x)$ is given by
\begin{equation*}
C_{J,\sigma}(x) = t^{r(r+1-2N)/2}  A_m(x) A_{J \times L}(x,x) \tau_{J} \bigl( \widetilde{A}_{J \times \sigma([m])}(x,x) \varPhi(\sigma)  K_\sigma( \Delta_{m}(x) ) \bigr).
\end{equation*}
with $L = [m+1,N] \setminus J$. We stress that $C_{J,\sigma}(x)=C_{J,\sigma w}(x)$
if $w \in  \mathfrak S_m \times  \mathfrak S_{m+1,N}$. As such, it makes sense to consider $[\sigma] \in  \mathfrak S_N/ ( \mathfrak S_m \times  \mathfrak S_{m+1,N})$.
\end{proposition}

\begin{proof}
From Lemma~\ref{eqPHI} and  the relation $K_{\sigma w}(\Delta_m(x)) = (-1)^{Z_w} K_\sigma (\Delta_{\{w(1),\ldots,w(m)\}}(x))$, we have immediately that $C_{J,\sigma}(x)=C_{J,\sigma w}(x)$
if $w \in  \mathfrak S_m \times  \mathfrak S_{m+1,N}$.

We now prove the central claim in the theorem. From Lemma~\ref{lemmasame}, it suffices to show that
\begin{align*}
&  e_{r}(Y_{m+1},\ldots,Y_{N}) \Delta_m^t(x) \frac{g(x)}{R_{[m] \times [m]}(x,y)}  \\
 & \qquad \qquad \qquad \qquad   =   \sum_{\substack{ J \subset [m+1,N] \\ |J| = r }} \, \, 
  \sum_{\substack{ [\sigma] \in  \mathfrak S_N/ ( \mathfrak S_m \times  \mathfrak S_{m+1,N}) \\ \sigma([m]) \cap L= \emptyset }} C_{J,\sigma} \tau_{J} K_\sigma \left( \frac{g(x)}{R_{[m] \times [m]}(x,y)} \right)
\end{align*}
  for every symmetric function $g(x)$.  Hence, the proposition will follow if we can prove that
\begin{align} 
&  e_{r}(Y_{m+1},\ldots,Y_{N})  \mathcal F_m(x,y) g(x)  \nonumber \\
   & \qquad \qquad \qquad \qquad =   \sum_{\substack{ J \subset [m+1,N] \\ |J| = r }}
 \, \, 
  \sum_{\substack{ [\sigma] \in  \mathfrak S_N/ ( \mathfrak S_m \times  \mathfrak S_{m+1,N}) \\ \sigma([m]) \cap L= \emptyset }} C_{J,\sigma}(x) \tau_{J} K_\sigma \left( \frac{g(x)}{R_{[m] \times [m]}(x,y)} \right)  \label{eqtoprove}
\end{align}
for every symmetric function $g(x)$.
The rest of the proof will be devoted to showing that \eqref{eqtoprove} holds.
Let $F(x,y):= e_r(Y_{m+1},\ldots, Y_N) \mathcal{F}_m(x,y) g(x)$.
Since $\mathcal S_{m+1,N}^t$ commutes with $F(x,y)$, we have by \eqref{eqHecke} that
\begin{equation*}
\mathcal S_{m+1,N}^t  \mathcal O(x,y) = [N-m]_t! \,  F(x,y)
\end{equation*}
or, equivalently, that
\begin{equation} \label{23.1}
F(x,y) = \dfrac{1}{[N-m]_t! }  \mathcal S_{m+1,N}^t e_r(Y_{m+1},\ldots, Y_N) \mathcal{F}_m(x,y) g(x) 
\end{equation}
Since $ \mathcal{F}_m(x,y) g(x) $ is symmetric in $x_{m+1},\dots,x_N$, we can use Lemma~\ref{lemmaeY} to rewrite $F(x,y)$ as
\begin{equation*}
  F(x,y) = \dfrac{1}{[N-m]_t! [N-r-m]_t! [r]_t!}  \mathcal S_{m+1,N}^t \sum_{\sigma \in  \mathfrak S_{m+1,N}} T_{\sigma} Y_{N-r+1} \cdots Y_N  \mathcal{F}_m(x,y) g(x) 
\end{equation*}
The relation $ \mathcal S_{m+1,N}^t  T_{\sigma}=t^{\ell(\sigma)} \mathcal S_{m+1,N}^t $ can then be used to get
\begin{equation*}
   F(x,y)  = \dfrac{1}{[N-r-m]_t! [r]_t!}  \mathcal S_{m+1,N}^t   Y_{N-r+1} \cdots Y_N  \mathcal{F}_m(x,y) g(x) 
\end{equation*}
It then follows by Lemma~\ref{rel Y} that 
\begin{equation*}
    F(x,y) = \dfrac{t^{(2m+r+1-2N)r/2}}{[N-r-m]_t! [r]_t!}  \mathcal S_{m+1,N}^t    \omega^{r} (\overline{T}_r \cdots \overline{T}_{m+r-1}) \cdots (\overline{T}_1 \cdots \overline{T}_{m})  \mathcal{F}_m(x,y) g(x) 
\end{equation*}
From Corollary~\ref{RelB}, we can use
$$\mathcal{F}_m(x,y) =\dfrac{(-1)^{\binom{m}{2}}}{\Delta_m(y)} \mathcal A^{(y)}_m \mathcal{NF}_m(x,y)$$
to deduce that
\begin{equation*}
    F(x,y) = \dfrac{(-1)^{\binom{m}{2}}}{\Delta_m(y)} \dfrac{t^{(2m+r+1-2N)r/2}}{[N-r-m]_t! [r]_t!}   \mathcal S_{m+1,N}^t     \omega^{r}  \mathcal A^{(y)}_m (\overline{T}_r \cdots \overline{T}_{m+r-1}) \cdots (\overline{T}_1 \cdots \overline{T}_{m})  \mathcal{NF}_m(x,y) g(x) 
\end{equation*}
Since $g(x)$ commutes with all $\bar T_i$'s we can use Lemma~\ref{TnegNF} to get
\begin{equation*}
    F(x,y) =  \frac{p(t)}{(-1)^{\binom{m}{2}}\Delta_m(y)} \mathcal S_{m+1,N}^t     \omega^{r}  \mathcal A^{(y)}_m  \mathcal{NF}_m^r(x,y) g(x) 
\end{equation*}
where, for simplicity,, we have set
$$p(t)= \dfrac{t^{r(r+1-2N)/2}}{[N-r-m]_t! [r]_t!} $$
Now, we will multiply and divide the quantity $\mathcal{NF}_m^r(x,y)$ by
 $R_{\mathcal{M}}(x,y)$, where 
$\mathcal{M}$ is the triangle with vertices $\{(r+2,m),(m+r,m),(m+r,2)\}$.
This way, the denominator $R_{\Bset_r'}$ becomes a rectangle and we have
\begin{equation*}
F(x,y) = \frac{p(t)}{(-1)^{\binom{m}{2}}\Delta_m(y)}  \mathcal S_{m+1,N}^t    \omega^r \mathcal{A}^{(y)}_m \dfrac{R_{\Bset_r} (x,ty) R_{\mathcal{M}}(x,y)}{R_{[m+r] \times [m]}(x,y)} g(x)
\end{equation*}
Observe that the rectangle
$[r] \times [m] \subseteq \Bset_r$ is such that $R_{[r] \times [m]}(x,ty)$ commutes with
$\mathcal{A}^{(y)}_m$.  It is also obvious that $R_{[m+r] \times [m]}$ commutes with
$\mathcal{A}^{(y)}_m$.  Hence
\begin{equation*}
F(x,y) = \frac{p(t)}{(-1)^{\binom{m}{2}}\Delta_m(y)}  \mathcal S_{m+1,N}^t   \omega^r \dfrac{R_{[r] \times [m]}(x,ty)}{R_{[m+r] \times [m]}(x,y)} \mathcal{A}^{(y)}_m  R_{\Bset_r\setminus([r] \times [m])}(x,ty) R_{\mathcal{M}}(x,y) g(x)
\end{equation*}
But $\Bset_r\setminus ([r] \times [m])$ is the triangle with vertices $(r+1,m-1),(r+1,1)$, and $(m+r-1,1)$.  We can thus use Proposition~\ref{RelR} to get
\begin{equation*}
F(x,y) = p(t)  \mathcal S_{m+1,N}^t   \omega^r \dfrac{R_{[r] \times [m]}(x,ty)}{R_{[m+r] \times [m]}(x,y)} \Delta^t_{\{r+1,\ldots,m+r\}}(x) g(x)
\end{equation*}
It will prove convenient to multiply and divide by 
$R_{[m+r+1,N] \times [m]}(x,y)$ so that $R_{[N] \times [m]}(x,y)$ appears in the denominator. This yields
\begin{equation*}
F(x,y) = p(t)  \mathcal S_{m+1,N}^t   \omega^r \dfrac{R_{[r] \times [m]}(x,ty)R_{[m+r+1,N] \times [m]}(x,y) }{R_{[N]\times[m]}(x,y)} \Delta^t_{\{r+1,\ldots,m+r\}}(x) g(x)
\end{equation*}
Applying $\omega^r$ (which amounts to the permutation that maps $j \mapsto j-r$ modulo $N$ followed by $\tau_{N-r+1,N}=\tau_{N-r+1} \tau_{N-r+2} \cdots \tau_N$) we obtain that
\begin{equation} \label{eq tauJ}
F(x,y) = p(t)  \mathcal S_{m+1,N}^t   \tau_{N-r+1,N} \dfrac{R_{[N-r+1,N] \times [m]}(x,ty)R_{[m+1,N-r] \times [m]}(x,y) }{R_{[N]\times[m]}(x,y)} \Delta^t_{m}(x) g(x)
\end{equation}
We should stress at this point that $F(x,y)/ \Delta^t_{m}(x)$ is both symmetric in $x_1,\dots,x_m$ and in $x_{m+1},\dots,x_{N}$.  This is because applying $\mathcal S_{m+1,N}^t$ ensures that the result is symmetric in $x_{m+1},\dots,x_{N}$ while the symmetry in $x_1,\dots,x_m$ is straightforward given that
$\Delta^t_{m}(x)$ commutes with $ \mathcal S_{m+1,N}^t   \tau_{N-r+1,N}$.

Now,  we need to use the expansion 
\begin{equation} \label{Kexpa}
 \mathcal S_{m+1,N}^t  = \left( \sum_{\substack{ \sigma \in  \mathfrak S_{m+1,N}} } K_\sigma \right)  \left(\prod_{m+1 \leq i<j \leq N} \dfrac{x_i-t x_j}{x_i-x_j} \right) 
\end{equation}
in \eqref{eq tauJ}. From the symmetry of $F(x,y)$, it will suffice to focus on the
term 
$\tau_{N-r+1,N}$ as the remaining terms $\tau_J$ for  $J \subseteq [m+1,N]$ and $|J|=r$ will be obtained by symmetry (only those terms can occur since
 $\mathcal S_{m+1,N}^t$ only contains $K_\sigma$'s such that $\sigma \in \mathfrak S_{m+1,N}$). For simplicity, we will let $J_0=[N-r+1,N]$ and $L_0=[m+1,N-r]$.
When we only focus on  the term $\tau_{J_0}=\tau_{N-r+1,N}$, we need to sum over the $\sigma$'s in \eqref{Kexpa} such that $\sigma (J_0)=J_0$. Observe that those permutations leave the expression to the right of $ \mathcal S_{m+1,N}^t$   invariant in \eqref{eq tauJ}.  Using  Lemma~\ref{RelU+} (in the case $J=[N-r+1,N]$ and with $[1,N]$ replaced by $[m+1,N-r]$) to obtain 
$$
\sum_{\substack{ \sigma \in  \mathfrak S_{m+1,N} \\ \sigma (J_0)=J_0 }} K_\sigma
\left(\prod_{m+1 \leq i<j \leq N} \dfrac{x_i-t x_j}{x_i-x_j} \right)
= [N-m-r]_t! [r]_t! A_{J_0\times L_0}(x,x)
$$
we thus conclude that the term in ${\tau_{J_0}}$ in $F(x,y)$ is given by
\begin{equation*}
t^{r(r+1-2N)/2} A_{J_0\times L_0}(x,x)  \tau_{J_0} \dfrac{R_{{J_0} \times [m]}(x,ty)R_{L_0 \times [m]}(x,y) }{R_{[N]\times[m]}(x,y)} \Delta^t_{m}(x) g(x)
\end{equation*}
By symmetry, we thus get that
\begin{equation*}
F(x,y) =t^{r(r+1-2N)/2} \sum_{J \subseteq [m+1,N] \, ; \, |J|=r}  A_{J \times L}(x,x)  \tau_J \dfrac{R_{J \times [m]}(x,ty)R_{L \times [m]}(x,y) }{R_{[N]\times[m]}(x,y)} \Delta^t_m(x) g(x)
\end{equation*}
where $L=[m+1,N] \setminus J$.

  Now, we want to expand $F(x,y)$ as
  $$
F(x,y) = \sum_{\substack{J\subseteq [m+1,N] \\ |J|=r}} \, \,  \sum_{[\sigma] \in  \mathfrak S_N/ ( \mathfrak S_m \times  \mathfrak S_{m+1,N}) } C_{J,\sigma}(x) \tau_J K_\sigma \left( \dfrac{g(x)}{R_{[m]\times [m]}(x,y)} \right)
  $$
for some coefficients $C_{J,\sigma}(x)$.
Since $g(x)$ is an arbitrary symmetric functions, the terms in $\tau_J g(x)$ need to be equal on both sides.  We thus have that
$$
t^{r(r+1-2N)/2} A_{J\times L}(q^{-1}x,x)   \dfrac{R_{{J} \times [m]}(x,ty)R_{L \times [m]}(x,y) }{R_{[N]\times[m]}(x,y)} \Delta^t_{m}(x) = \sum_{\sigma \in  \mathfrak S_{m,N}} \left(\tau_J^{-1} C_{J,\sigma}(x) \right) K_\sigma \left( \dfrac{1}{R_{[m]\times [m]}(x,y)} \right)
$$
The coefficient $\tau_J^{-1} C_{J,w}(x)$ can be obtained by multiplying by $K_w\bigl(R_{[m] \times [m]}(x,y) \bigr)$ and then taking the specialization $y_i=x^{-1}_{w(i)}$ for $i=1,\dots,m$ (this way, all the terms such that $\sigma \neq w$ cancel on the r.h.s.). Hence
$$
\tau_J^{-1} C_{J,w}(x)= t^{r(r+1-2N)/2} A_{J\times L}(q^{-1}x,x)   \dfrac{R_{{J} \times [m]}(x,ty)R_{L \times [m]}(x,y) }{R_{[N]\times[m]}(x,y)} \Delta^t_{m}(x)  K_w\bigl(R_{[m] \times [m]}(x,y) \bigr) \Bigg|_{y_i=x_{w(i)}^{-1}}
$$
or, equivalently,
\begin{equation} \label{lastbig}
\tau_J^{-1} C_{J,w}(x)= t^{r(r+1-2N)/2} A_{J\times L}(q^{-1}x,x)   \dfrac{R_{{J} \times [m]}(x,ty)R_{L \times [m]}(x,y) }{R_{([N]\setminus w([m]))\times[m]}(x,y)} \Delta^t_{m}(x)   \Bigg|_{y_i=x_{w(i)}^{-1}}
\end{equation}
When considering $y_i = x_{w(i)}^{-1}$, the following holds:
 \begin{equation*}
R(x_i,ay_j)= \left\{ \begin{array}{lcc}
             -\left(\dfrac{ax_i - x_{\mu(j)}}{x_{\mu(j)}}\right) &   {\rm if}  & i \neq w(j) \\
             \\ -(a-1) &  {\rm if} & i = w(j)
             \end{array}
   \right.
 \end{equation*}
The extra sign that the specialization generates on the r.h.s. of  \eqref{lastbig} is then
\begin{equation*}
   \#(J \times [m]) + \#(L \times [m]) + \#\bigl((([N]\setminus w([m])) \times [m]\bigr)  
     = \#J \cdot m + \#L \cdot m + (N-m)\cdot m 
\end{equation*}
which is equal to $2(N-m)m$.  The extra sign, being even, can thus be ignored.

We now split the set $J \times [m]$ as the disjoint union of $G_1$ and $G_2$, where
$$
G_1 = \{(i,j) \in J \times [m] \ |\ i \neq w(j)\} \quad {\rm and} \quad  G_2 = \{(i,j) \in J \times [m] \ |\ i = w(j)\}
$$
Hence, after  multiplying and dividing the r.h.s. of \eqref{lastbig} by $R_{G_1}(x,y)$, we obtain 
$$
\tau_J^{-1} C_{J,w}(x)= t^{r(r+1-2N)/2} A_{J\times L}(q^{-1}x,x)
\dfrac{R_{G_1}(x,ty)}{R_{G_1}(x,y)} R_{G_2}(x,ty)\dfrac{R_{G_1}(x,y)R_{L \times [m]}(x,y) }{R_{([N]\setminus w([m]))\times[m]}(x,y)} \Delta^t_{m}(x)   \Bigg|_{y_i=x_{w(i)}^{-1}}
$$
It is easy to check that
$$
  \dfrac{R_{G_1}(x,ty)}{R_{G_1}(x,y)} \Bigg |_{y_i=x_{w(i)}^{-1}}=  \widetilde A_{J \times w([m])}(x,x),  \quad   R_{G_2}(x,ty) \Big |_{y_i=x_{w(i)}^{-1}}= (t-1)^{\# (J \cap w([m]))}
$$
as well as
$$
R_{G_1}(x,y)  \Big |_{y_i=x_{w(i)}^{-1}} =  \dfrac{\widetilde{\varDelta}_{J \times w([m])}(x,x)  }{(x_{w(1)} \cdots x_{w(m)})^{\#J}}  \prod_{i \in J \cap w([m])} { x_i} \, , \quad
 R_{L \times [m]}(x,y) \Big |_{y_i=x_{w(i)}^{-1}} =    \dfrac{\widetilde{\varDelta}_{L \times w([m])}(x,x)  }{(x_{w(1)} \cdots x_{w(m)})^{\#L}} 
$$
and
$$
R_{([N] \setminus w([m])) \times [m]}(x,y)  \Big |_{y_i=x_{w(i)}^{-1}} =    \dfrac{ \varDelta_{([N]\setminus w([m])) \times w([m])}(x,x) }{(x_{w(1)} \cdots x_{w(m)})^{N-\#L-\#J}} 
$$
Hence, using $J \cap w([m])=\BB_w^c$, we obtain
%
    


\begin{equation*}
\tau_J^{-1} C_{J,w}(x) =  t^{r(r+1-2N)/2} A_{J\times L}(q^{-1}x,x)   \widetilde{A}_{J \times w([m])} \left( \displaystyle \prod_{i \in \BB_w^c} { x_i} (t-1) \right)  \dfrac{\widetilde{\varDelta}_{J \times w([m])} \varDelta_{L \times w([m])} }{ \varDelta_{([N]\setminus w([m])) \times w([m])}}  \Delta^t_{m}(x) 
\end{equation*}
where the dependency is always in the variables $(x,x)$ when not specified.
We deduce immediately that $C_{J,w}=0$ whenever $L \cap w([m])\neq \emptyset$ since  $\varDelta_{L \times w([m])}(x,x)=0$ in that case.  Finally, using
 $[N] = \AA_{w} \cup \AA_{w}^c \cup \BB_{w} \cup \BB_{w}^c$ and $w([m]) = \AA_{w} \cup \BB_{w}^c$, we get that
\begin{equation*}
\tau_J^{-1} C_{J,w}(x) =  t^{r(r+1-2N)/2} A_{J\times L}(q^{-1}x,x)  \widetilde{A}_{J \times w([m])} \left( \displaystyle \prod_{i \in B_\mu^c} x_i(t-1) \right) \dfrac{\widetilde{\varDelta}_{\BB^c_{w} \times w([m])} }{ \varDelta_{ \AA^c_{w} \times w([m])}}  \Delta^t_{m}
\end{equation*}
when $L \cap w([m])= \emptyset$.  From  Lemma~\ref{RelDet}, this implies that
\begin{equation*}
C_{J,w}(x) =   t^{r(r+1-2N)/2} A_m(x) A_{J\times L}(x,x)  \tau_{J} \bigl( \widetilde{A}_{J \times w([m])}(x,x) \varPhi(w) K_w(\Delta_{m}) \bigr).
\end{equation*}
when $L \cap w([m])= \emptyset$. This proves \eqref{eqtoprove} and the proposition thus holds.
\end{proof}

\section{Pieri rules}
Before proving the Pieri rules for the bisymmetric Macdonald polynomials, we first need to establish a crucial lemma.

We will say that a composition $(\Lambda_1,\dots,\Lambda_N)$ is biordered if
$\Lambda_1\geq \Lambda_2 \geq \cdots \geq \Lambda_m$ and $\Lambda_{m+1}\geq \Lambda_2 \geq \cdots \geq \Lambda_N$.  Note that if $\Lambda$ is not biordered then there exists a permutation $\sigma \in  \mathfrak S_m \times  \mathfrak S_{m+1,N}$ such that $\sigma \Lambda$ is biordered.  For $J\subseteq [N]$, we will also let $\tau_J \Lambda= \Lambda + \varepsilon^J$, where $\varepsilon^J_i=1$ if $i \in J$ and $0$ otherwise.
\begin{lemma} \label{lemmaOmega} Suppose that
 $\sigma\in \mathfrak S_N$ and
  $J \subseteq [m+1,N]$ are such that $\sigma([m]) \cap L=\emptyset$, where we recall that $L=[m+1,N]\setminus J$.
  Let $(\Lambda,w)$ generate a superevaluation, and suppose that the composition  $\Omega=\sigma^{-1}\tau_J(\Lambda+(1^m))-(1^m)$ is biordered. The following holds:
\begin{enumerate} 
\item
  If $(\Omega,w\sigma)$  does not generate a superevaluation
  then $u_\Lambda^+(C_{J,\sigma})=0$, where $C_{J,\sigma}(x)$ is such as defined in Proposition~\ref{Op+}.
   \item 
     Suppose that $(\Omega,w\sigma)$  
     generates a superevaluation.  If $\delta \in \mathfrak S_N$ is also such that $(\Omega,w\delta)$  
     generates a superevaluation then    
   $\sigma(\mathfrak S_m \times \mathfrak S_{m+1,N})=\delta(\mathfrak S_m \times \mathfrak S_{m+1,N})$ in $\mathfrak S_N/(\mathfrak S_m \times \mathfrak S_{m+1,N})$.
 \item If 
   $I \subseteq [m+1,N]$  is such that $\Omega=\sigma^{-1}\tau_I(\Lambda+(1^m))-(1^m)$, then $I=J$.
 \end{enumerate} 
\end{lemma}

\begin{proof}
  We first show that ${\it (1)}$ holds.  Suppose first that $\Omega$
  is not a superpartition.  Given that $\Omega$ is biordered, this can only happen if $\Omega_a=\Omega_{a+1}$
for a given $a\in [m-1]$, which can be visualized as
\\
\begin{equation}
{\small \tableau[scY]{ \bl a & & & & \bl \tcercle{} \\  \bl b & & & &\bl \tcercle{} }}
\end{equation}
\\
with $b=a+1$.  Now, $\Omega=\sigma^{-1}\tau_J(\Lambda+(1^m))-(1^m)$ translates in coordinates to 
\begin{equation}   \label{OL}
  \bigl(\Omega+(1)^m\bigr)_a = \bigl(\Lambda + (1^m)\bigr)_{\sigma(a)}+\varepsilon_{\sigma(a)}^J
\end{equation}
where $\varepsilon_{i}^J=1$ if $i \in J$ and 0 otherwise. 
Hence there are two possible cases:
(i) $\sigma(a), \sigma(b) \in [m+1,N]$ or (ii) $\sigma(a) \in [m], \sigma(b) \in [m+1,N]$
(the case  $\sigma(a) \in [m+1,N], \sigma(b) \in [m]$ is equivalent).

Consider first the case (i).   We have that $\sigma(a), \sigma(b) \in J$ since
$\sigma([m]) \cap L=\emptyset$  by hypothesis.
We thus deduce from \eqref{OL} that $\Omega_a=\Lambda_{\sigma(a)}$ and $\Omega_b=\Lambda_{\sigma(b)}$, which implies that $\Lambda_{\sigma(a)}=\Lambda_{\sigma(b)}$.
This in turn implies that the permutation $w$ can be chosen such that $w\sigma(b)=w\sigma(a)+1$, in which case we will have $\Lambda^\circledast_{w\sigma(a)}=\Lambda_{\sigma(a)}+1$ and $\Lambda^\circledast_{w\sigma(b)}=\Lambda_{\sigma(b)}+1$. 
Hence the term $\widetilde A_{J \times \sigma([m])}(x,x)$ in $C_{J,\sigma}(x)$ contains a factor $A_{\sigma(b),\sigma(a)}(x)$ such that
$$
u_\Lambda^+(\tau_J A_{\sigma(b),\sigma(a)}(x))= u_\Lambda^+\left( \frac{qtx_{\sigma(b)}-qx_{\sigma(a)}}{qx_{\sigma(b)}-qx_{\sigma(a)}} \right)=  \frac{q^{\Lambda_{\sigma(b)}+2}t^{2-w\sigma(b)}-q^{\Lambda_{\sigma(a)}+2}t^{1-w\sigma(a)}}{q^{\Lambda_{\sigma(b)}+2}t^{1-w\sigma(b)}-q^{\Lambda_{\sigma(a)}+2}t^{1-w\sigma(a)}}=0
$$
and thus  $C_{J,\sigma}(x)$  vanishes in that case. 

The case (ii) is almost identical.   We have that $\sigma(a) \in [m]$ and  $\sigma(b) \in J$ since $\sigma([m]) \cap L=\emptyset$  by hypothesis.
\\
\begin{equation*}
{\footnotesize\tableau[scY]{ \bl \sigma(a) & \bl & & & & \bl \tcercle{} \\ \bl & \bl  & &  \bl \tcercle{} \\ \bl & \bl & \bl & \bl \vspace{-2ex}\vdots \\\bl &  \bl & \bl \\  \bl \sigma(b)& \bl & & &   }}
\end{equation*}
\\
We thus deduce from \eqref{OL} that $\Omega_a=\Lambda_{\sigma(a)}$ and $\Omega_b=\Lambda_{\sigma(b)}$, which implies that $\Lambda_{\sigma(a)}=\Lambda_{\sigma(b)}$.
This in turn implies that the permutation $w$ can be chosen such that $w\sigma(b)=w\sigma(a)+1$, in which case we will have $\Lambda^\circledast_{w\sigma(a)}=\Lambda_{\sigma(a)}+1$ and $\Lambda^\circledast_{w\sigma(b)}=\Lambda_{\sigma(b)}$. 
Hence the term $\widetilde A_{J \times \sigma([m])}(x,x)$ in $C_{J,\sigma}(x)$ contains a factor $A_{\sigma(b),\sigma(a)}(x)$ such that
$$
u_\Lambda^+(\tau_J A_{\sigma(b),\sigma(a)}(x))= u_\Lambda^+\left( \frac{qtx_{\sigma(b)}-x_{\sigma(a)}}{qx_{\sigma(b)}-x_{\sigma(a)}} \right)=  \frac{q^{\Lambda_{\sigma(b)}+1}t^{2-w\sigma(b)}-q^{\Lambda_{\sigma(a)}+1}t^{1-w\sigma(a)}}{q^{\Lambda_{\sigma(b)}+1}t^{1-w\sigma(b)}-q^{\Lambda_{\sigma(a)}+1}t^{1-w\sigma(a)}}=0
$$
and thus  $C_{J,\sigma}(x)$ also vanishes in that case. 

We now have to show that $C_{J,\sigma}(x)=0$ when any of the two following cases occurs:
\begin{enumerate}
    \item $w \sigma (\Omega + (1^m)) \neq \Omega^{\circledast} $
    \item $w \sigma (\Omega)  \neq \Omega^{*} $
\end{enumerate}
In the case (1), we have
$$\Omega^{\circledast} \neq w \sigma (\Omega + (1^m)) = w \tau_J(\Lambda + (1^m)) = \tau_{w(J)} \Lambda^{\circledast}$$
This can only happen if  $\tau_{w(J)} \Lambda^{\circledast}$ is not a partition, that is, if we have the following situation:
\\
\begin{equation*}
{\small \tableau[scY]{ \bl b & & & \\  \bl a & & & &\fl}}
\end{equation*}
\\
where $1+b=a$, $\Lambda^{\circledast}_a = \Lambda^{\circledast}_b $, $a \in w(J)$ and $b\not \in w(J)$. Note that $b$ cannot belong to $w([m])$ since otherwise
the diagram of $\Lambda^*$ would be of the form 
\\
\begin{equation}
{\small \tableau[scY]{ \bl b & &  & \bl   \\  \bl a & & & }}
\end{equation}
\\
and thus not a partition.
We therefore conclude that $b \in w(L)$. Hence there exist
$j \in J, l \in L$ such that $a = w(j)$ and $b = w(l)$,
which implies that the term $A_{J \times L}(x,x)$ in $C_{J,\sigma}(x)$ contains a factor $A_{j,l}(x)$ such that
$$
u_\Lambda^+(A_{j,l}(x))= u_\Lambda^+\left( \frac{tx_{j}-x_{l}}{x_j-x_{l}} \right)   = \dfrac{q^{\Lambda^{\circledast}_a}t^{2-a}-q^{\Lambda^{\circledast}_b}t^{1-b}}{q^{\Lambda^{\circledast}_a}t^{1-a}-q^{\Lambda^{\circledast}_b}t^{1-b}} = 0
$$
as wanted.

We finally need to consider case (2) which amounts to
\begin{equation} \label{eq(2)}
  \Omega^{*} \neq w \sigma (\Omega + (1^m)-(1^m)) = w (\tau_J(\Lambda + (1^m))- \mu (1^m)) = \tau_{w(J)} \Lambda^{\circledast}-w \sigma (1^m)
\end{equation}  
From Case (1) we know that  $C_{J,\sigma}(x)=0$ if    $\tau_{w(J)} \Lambda^{\circledast}$ is not a partition from which we can  suppose that $\tau_{w(J)} \Lambda^{\circledast}$ is a partition. Hence \eqref{eq(2)} will hold in the two following situations:
\\
\begin{equation} \label{twocases}
  {\small \tableau[scY]{ \bl b & & & X \\  \bl a & & & }}
\qquad {\rm and} \qquad   {\small \tableau[scY]{ \bl b & & & & X  \\  \bl a & & & & \fl }}
\end{equation}
\\
where $X$ stands for a removed cell (the diagrams of $\Lambda^\circledast$ are those without $X$'s and black square).  We first show that the case to the left cannot occur.  Indeed, we have in that case that
\begin{enumerate}
\item[{\it a})] $b \not \in w(J)$ 
\item[{\it b})] $b \in w \sigma ([m])$ 
\item[{\it c})] $b \not \in w ([m])$  (otherwise $\Lambda$ would not be a superpartition)
\end{enumerate}  
From {\it b}) and {\it c}) we deduce that $b \not \in w (\sigma([m]) \cap [m])$. 
Therefore  $b \in w (J)$ since $\sigma([m]) \cap L=\emptyset$ by hypothesis. But this contradicts   {\it a}).

Finally, we consider the case to the right in \eqref{twocases}. We have
$1+b=a$, $\Lambda^{\circledast}_a+1 = \Lambda^{\circledast}_b $,  $b \in w \sigma ([m])$ and  $a \in w (J)$.  Therefore, there exist  $j \in J, s \in \sigma([m])$
such that  $a = w(j)$ and  $b = w(s)$.
The term $\widetilde A_{J \times \sigma([m])}(x,x)$ in $C_{J,\sigma}(x)$ thus contains a factor $A_{j,s}(x)$ such that
$$
u_\Lambda^+(\tau_J A_{j,s}(x))= u_\Lambda^+\left( \frac{qtx_{j}-x_{s}}{qx_{j}-x_{s}} \right)= \dfrac{q^{\Lambda^{\circledast}_a+1}t^{2-a}-q^{\Lambda^{\circledast}_b}t^{1-b}}{q^{\Lambda^{\circledast}_a}t^{1-a}-q^{\Lambda^{\circledast}_b}t^{1-b}} = 0
$$
which completes the proof of part ${\it (1)}$ of the Lemma.

Part ${\it (2)}$ and ${\it (3)}$ of the lemma are much simpler to prove.
We start with (2).  Since both $(\Omega,w\delta)$ and  $(\Omega,w\sigma)$ generate a superevaluation, we have that
$$\Omega^* = w \delta \Omega=w\sigma \Omega \qquad  {\rm and} \qquad  \Omega^\circledast = w \delta (\Omega+(1^m))=w\sigma (\Omega+(1^m))
$$
Hence, $\sigma^{-1} \delta \Omega= \Omega$ and $\sigma^{-1} \delta (\Omega+(1^m))= (\Omega+(1^m))$.  We thus conclude that 
$\sigma^{-1} \delta \in \mathfrak S_m \times \mathfrak S_{m+1,N}$ or equivalently, that $\sigma(\mathfrak S_m \times \mathfrak S_{m+1,N})=\delta(\mathfrak S_m \times \mathfrak S_{m+1,N})$, as wanted.

As for ${\it (3)}$, we have
$$
\Omega=\sigma^{-1}\tau_J(\Lambda+(1^m))-(1^m)=\sigma^{-1}\tau_I(\Lambda+(1^m))-(1^m) \iff \tau_{I}^{-1}\tau_J(\Lambda+(1^m))=\Lambda+(1^m)) 
$$
which implies that $I=J$.
\end{proof}
We can now state our main theorem.  It is important to note that a more explicit characterization 
of the indexing superpartitions appearing in the Pieri rules  will be provided in Corollary~\ref{coroPieri}. Also recall that the evaluation $u_{\Lambda_0^+}(\mathcal P_\Lambda)$ was given in \eqref{ev+mac}.
\begin{theorem} \label{theoPieri} For $r\in \{1,\dots,N-m\}$, the bisymmetric Macdonald polynomial
$\mathcal P_{\Lambda}(x;q,t)$ obeys the following Pieri rules
\begin{equation*}
  e_{r}(x_{m+1},\dots,x_N)    {\mathcal P}_{\Lambda}(x;q,t) = \sum_\Omega \left(\frac{u_\Lambda^+(C_{J,\sigma})}{ u_\Lambda^+(\Delta_m^t)} \frac{u_{\Lambda_0^+}(\mathcal P_\Lambda)}{u_{\Lambda_0^+}(\mathcal P_\Omega)} \right)   {\mathcal P}_{\Omega}(x,q,t)  
\end{equation*}
where the coefficients $C_{J,\sigma}(x)$ were obtained explicitly in Proposition~\ref{Op+} and 
where the sum is over all superpartitions $\Omega$ such that there exists a
$\sigma \in \mathfrak S_N$ and
a $J \subseteq [m+1,N]$ of size $r$ such that
\begin{itemize}
\item  $\sigma(\Omega+(1^m))=\tau_J(\Lambda+(1^m))$ 
\item $\sigma([m]) \cap L = \emptyset$, where $L=[m+1,\dots,N]\setminus J$ 
\item  $(\Omega,w\sigma)$ is a superevaluation, where $w$ is such that $(\Lambda,w)$ generates a superevaluation
\end{itemize}
\end{theorem}
\begin{proof} We know from Theorem~\ref{Op+} that 
\begin{equation*}
  e_{r}(Y_{m+1},\ldots,Y_{N}) \Delta_m^t(x)  \tilde {\mathcal P}^+_{\Psi}(x;q,t) =  \sum_{\substack{ J \subset [m+1,N] \\ |J| = r }} \, \, 
  \sum_{\substack{ [\sigma] \in  \mathfrak S_N/ ( \mathfrak S_m \times  \mathfrak S_{m+1,N}) \\ \sigma([m]) \cap L= \emptyset }} C_{J,\sigma}(x) \tau_{J} K_\sigma  \tilde {\mathcal P}^+_{\Psi}(x;q,t)
\end{equation*}
where we recall that $\tilde {\mathcal P}^+_{\Psi}(x;q,t)$ was defined in Theorem~\ref{SimMac}.
Using $e_r^{(m+1)}$ to denote $e_{r}(x_{m+1},\dots,x_N)$, we obtain from
Lemma~\ref{AutVal} that
\begin{equation*}
  u_\Psi^+ (e_{r}^{(m+1)} ) \Delta_m^t(x)  \tilde {\mathcal P}^+_{\Psi}(x;q,t) =  \sum_{\substack{ J \subset [m+1,N] \\ |J| = r }} \, \, 
  \sum_{\substack{ [\sigma] \in  \mathfrak S_N/ ( \mathfrak S_m \times  \mathfrak S_{m+1,N}) \\ \sigma([m]) \cap L= \emptyset }} C_{J,\sigma}(x)\tau_{J} K_\sigma  \tilde {\mathcal P}^+_{\Psi}(x;q,t) 
\end{equation*}
Let $\Lambda$ be a superpartition such that $(\Lambda,w)$ is a superevaluation.
Applying $u_\Lambda^+$ on both sides of the equation (and dropping the dependencies in $x$ in the evaluations for simplicity) leads to
\begin{equation} \label{equ+}
   u_\Psi^+(e_{r}^{(m+1)} )  u_\Lambda^+ (\Delta_m^t)  u_\Lambda^+ ( \tilde {\mathcal P}^+_{\Psi}) =  \sum_{\substack{ J \subset [m+1,N] \\ |J| = r }} \, \, 
  \sum_{\substack{ [\sigma] \in  \mathfrak S_N/ ( \mathfrak S_m \times  \mathfrak S_{m+1,N}) \\ \sigma([m]) \cap L= \emptyset }} u_\Lambda^+(C_{J,\sigma}) u_\Lambda^+(\tau_{J} K_\sigma  \tilde {\mathcal P}^+_{\Psi}) 
\end{equation}
Now, in  $u_\Lambda^+(\tau_{J} K_\sigma  \tilde {\mathcal P}_{\Psi})$, the evaluation amounts to the following substitution
$$
x_i = q^{(\Lambda +(1^m))_{\sigma(i)}+\varepsilon_{\sigma(i)}^J}t^{1-w\sigma(i)}
$$
where again $\varepsilon_{i}^J=1$ if $i\in J$ and 0 otherwise.  Comparing with \eqref{OL}, we have that the substitution is
$$
x_i = q^{(\Omega +(1^m))_i}t^{1-w\sigma(i)}
$$
where $\Omega=\sigma^{-1}\tau_J(\Lambda+(1^m))-(1^m)$.  Choosing $\sigma$ in $[\sigma]$ such that $\Omega$ is biordered, we deduce from 
Lemma~\ref{lemmaOmega} that for $u_\Lambda^+(C_{J,\sigma})$ not to vanish, 
we need $(\Omega,w\sigma)$ to generate a superevaluation (and in particular for $\Omega$ to be a superpartition).  We also get from Lemma~\ref{lemmaOmega} {\it 2)} and {\it 3)} that the  superpartition  $\Omega$ can arise in at most one way in the sums in the r.h.s. of \eqref{equ+}.  As such, we obtain that
\begin{equation*}
  u_\Psi^+(e_{r}^{(m+1)}) u_\Lambda^+(\Delta_m^t)  u_\Lambda^+ (  \tilde {\mathcal P}^+_{\Psi}) = \sum_\Omega u_\Lambda^+(C_{J,\sigma})  u_\Omega^+ ( \tilde {\mathcal P}^+_{\Psi} )
\end{equation*}
where the sum is over all superpartitions $\Omega$ such that there exists a
$\sigma \in \mathfrak S_N$ and
a $J \subseteq [m+1,N]$ of size $r$ such that $\sigma(\Omega+(1^m))=\tau_J(\Lambda+(1^m))$, such that $\sigma([m]) \cap L = \emptyset$,
and  such that $(\Omega,w\sigma)$ is a superevaluation.

The symmetry established in Theorem~\ref{SimMac} then implies that 
\begin{equation*}
  u_\Psi^+\bigl(e_{r}^{(m+1)} u_\Lambda^+(\Delta_m^t)   \tilde {\mathcal P}^+_{\Lambda}\bigr) = u_\Psi^+\left(\sum_\Omega u_\Lambda^+(C_{J,\sigma})   \tilde {\mathcal P}^+_{\Omega} \right) 
\end{equation*}
Now, the previous equation holds for every superpartition $\Psi$.  Therefore,
\begin{equation*}
  e_{r}(x_{m+1},\dots,x_N) u_\Lambda^+(\Delta_m^t)   \tilde {\mathcal P}^+_{\Lambda}(x;q,t) = \sum_\Omega u_\Lambda^+(C_{J,\sigma})   \tilde {\mathcal P}^+_{\Lambda}(x,q,t)  
\end{equation*}
from which we finally obtain that
\begin{equation*}
  e_{r}(x_{m+1},\dots,x_N)    {\mathcal P}_{\Lambda}(x;q,t) = \sum_\Omega \left(\frac{u_\Lambda^+(C_{J,\sigma})}{ u_\Lambda^+(\Delta_m^t)} \frac{u_{\Lambda_0^+}(\mathcal P_\Lambda)}{u_{\Lambda_0^+}(\mathcal P_\Omega)} \right)   {\mathcal P}_{\Omega}(x,q,t)  
\end{equation*}
\end{proof}

\section{Pieri rules and vertical strips}
In this section, we will give explicitly which superpartition $\Omega$ appear in the Pieri rules of Theorem~\ref{theoPieri}.  They will turn out to be certain vertical strips.

For partitions $\lambda$ and $\mu$, we say that $\mu/\lambda$ is a vertical $r$-strip if $|\mu|-|\lambda|=r$ and
$\mu_i-\lambda_i \in \{0,1 \}$ for all $i$,  where we consider that $\mu_i=0$ (resp. $\lambda_i=0)$ if $i$ is larger than the length of $\mu$ (resp. $\lambda$).

Given the superpartitions $\Lambda$ and $\Omega$, we say that $\Omega/\Lambda$ is a vertical $r$-strip if both $\Omega^{\circledast}/\Lambda^{\circledast}$ and $\Omega^{*}/\Lambda^{*}$ are vertical $r$-strips.  When  describing the vertical strip $\Omega / \Lambda$ with Ferrers' diagram, we will use the following notation:
\begin{itemize}
    \item the squares of $\Lambda$ will be denoted by $\WSS$
    \item the squares of $\Omega / \Lambda$ that do not lie over a circle of $\Lambda$ will be denoted by $\BSS$
    \item the squares of $\Omega / \Lambda$ that lie over a circle of $\Lambda$ will be denoted by $\OSS$
    \item the circles of $\Lambda$ that are still circles in $\Omega$ will be denoted by $\WCC$
    \item the circles of $\Omega$ that were not circles in $\Lambda$ will be denoted by $\BCC$
\end{itemize}
For instance, if $\Lambda = (5,3,1;4,3)$ and $\Omega = (5,4,0;5,4,2)$ the cells of the vertical 4-strip $\Omega / \Lambda$ are represented as:
\begin{equation*} 
{\small \tableau[scY]{  &&&& & \bl \cercle{}\\ &&& & \fl \\  &&& \tf \ocircle & \bl \cerclep \\ &&& \fl  \\ & \tf \ocircle{} \\  \bl \cerclep}}
\end{equation*}  

A row in the diagram of $\Omega / \Lambda$ that contains a   $\BSS$
will be called a $\BSS$-row (and similarly for $\WCC$, $\BCC$ and $\OSS$).
A row  that both contains a $\BCC$ and a $\OSS$ will be called a
${\footnotesize\tableau[scY]{ \tf \ocircle & \bl \cerclep} }$-row.
For instance, in our previous example, the set of $\BSS$-rows is $\{2,4\}$,
the set of  $\BCC$-rows is $\{3,6\}$, the set of  $\OSS$-rows is  $\{3,5\}$
while the set of ${\footnotesize\tableau[scY]{ \tf \ocircle & \bl \cerclep} }$-rows is  $\{3\}$.

  \begin{definition}  We will say that 
    $\Omega/ \Lambda$ is a vertical $r$-strip of type I if
\begin{enumerate}
   \item  $\Omega/ \Lambda$ is a vertical $r$-strip
   \item there are no
     ${\footnotesize\tableau[scY]{ \tf \ocircle & \bl \cerclep} }$-rows in the diagram of   $\Omega/ \Lambda$
\end{enumerate}
  \end{definition}
For instance, if $\Lambda=(3,1;5,4,3)$ and $\Omega=(4,0;6,4,3,2)$ then 
   $\Omega/ \Lambda$ is a vertical 3-strip of type I.
\begin{equation}
{\small \tableau[scY]{  &&&&&\fl \\  &&& & \bl \cerclep \\   &&&  \tf \ocircle \\  &&  \\  & \tf \ocircle  \\  \bl \cerclep   }}
\end{equation}
We first show that the $\Omega$'s that can appear in the Pieri rules of 
 Theorem~\ref{theoPieri} are such that  $\Omega / \Lambda$ is a vertical $r$-strip of type I.
\begin{lemma} \label{cond2vert} Let $\sigma$ and $J$ be such
  as in Theorem~\ref{theoPieri}, that is, such that
\begin{enumerate}
    \item $\sigma(\Omega+(1^m))=\tau_J(\Lambda+(1^m))$
    \item $\sigma([m]) \cap L = \emptyset $ con $L=[m+1,N]-J$
    \item $(\Omega,w\sigma)$ is a superevaluation if $(\Lambda,w)$ is a superevaluation.
    \item $J \subseteq [m+1,N]$ with  $|J|=r$.
\end{enumerate}
Then $\Omega / \Lambda$ is a vertical $r$-strip of type I.  
\end{lemma}
\begin{proof}
Applying $w$ on both sides of {\it (1)} gives
$w\sigma(\Omega+(1^m))=\tau_{w(J)}w\bigl(\Lambda+(1^m)\bigr)$.
From {\it (3)} we then get that
$\Omega^{\circledast} = \tau_{w(J)} \Lambda^{\circledast}$,
which immediately implies that $\Omega^{\circledast} / \Lambda^{\circledast}$ is a vertical $r$-strip.

Subtracting $(1^m)$ on both sides of  {\it (1)} gives
$\sigma(\Omega+(1^m))-(1^m)=\tau_{J}(\Lambda+(1^m))-(1^m)= \tau_J \Lambda$.
Applying again $w$ on both sides of the equation then yields
$w\sigma(\Omega)+w\sigma(1^m)-w(1^m)=\tau_{w(J)}(w\Lambda)$,
which from {\it (3)} amounts to
$\Omega^{*} +w\sigma(1^m)-w(1^m)=\tau_{w(J)}\Lambda^{*}$, or equivalently, to
$$\Omega^{*}=\tau_{w(J)}\Lambda^{*} +w(1^m)-w\sigma(1^m)$$
Note that by the action of the symmetric group on vectors, $w(1^m)$
adds a 1 in the positions $w([m])$ (and similarly for $w\sigma(1^m)$).
From {\it (4)}, we have that $w(J)\cap w([m]) = \emptyset$, which gives 
$\Omega^*_i -\Lambda_i^* \leq 1$.  Moreover, from {\it (2)}, we have that
$\sigma([m]) \subseteq [m] \cup J$, which implies that
$w\sigma([m]) \subseteq w([m]) \cup wJ$.  Hence $0 \leq \Omega^*_i -\Lambda_i^*  \leq 1$ and we have that $\Omega^*/\Lambda^*$ is a vertical $r$-strip as well.

Finally, suppose that row $i$ in $\Omega/\Lambda$ is a 
${\footnotesize\tableau[scY]{ \tf \ocircle & \bl \cerclep} }$-row.
We have in this case that $i \in \Omega^{\circledast} / \Lambda^{\circledast}$
as well as $i \in w([m])$ since $\Lambda$ has a circle in row $i$.  
But, as we have seen,  $\Omega^{\circledast} = \tau_{w(J)} \Lambda^{\circledast}$.
We thus have that  $i \in w(J) \cap w([m])$, which contradicts ${\it (4)}$. 
\end{proof}

\begin{remark} \label{sigmatilde}
Observe that in a vertical $r$-strip, the rows of 
$\Omega^* / \Lambda^*$ correspond to the $\OSS$-rows together with the $\BSS$-rows. Similarly, the rows of $\Omega^{\circledast} / \Lambda^{\circledast}$ correspond in a vertical strip to the  $\BCC$-rows together with the $\BSS$-rows.  By this observation, if
$\Omega/\Lambda$ is a vertical $r$-strip, then the number of 
$\OSS$-rows is equal to the number of  $\BCC$-rows.
\end{remark}
We now show that all $\Omega$'s such that  $\Omega / \Lambda$ is a vertical $r$-strip of type I do in fact appear in the Pieri rules of 
 Theorem~\ref{theoPieri}.
\begin{lemma} \label{vert2cond} Given $\Omega / \Lambda$ a vertical $r$-strip of type I, let $\tilde \sigma$ be any permutation that interchanges the
  $\OSS$-rows and the  $\BCC$-rows while leaving the remaining rows invariant (such a permutation can be defined by Remark~\ref{sigmatilde}).  Let also $\tilde J$ be the set of
  $\OSS$-rows and $\BSS$-rows.  If 
  $$\sigma = w^{-1} \tilde{\sigma} w  \quad {\rm   and   } \quad J = w^{-1} \tilde{\sigma} (\tilde{J})$$
  then there exists a permutation $s\in \mathfrak S_m \times \mathfrak S_{m+1,N}$
such that $\sigma'=\sigma s$ obeys the following relations:
\begin{enumerate}
    \item $\sigma'(\Omega+(1^m))=\tau_J(\Lambda+(1^m))$
    \item $\sigma'([m]) \cap L = \emptyset $ con $L=[m+1,N]-J$
    \item $(\Omega,w\sigma')$ is a superevaluation if $(\Lambda,w)$ is a superevaluation.
    \item $J \subseteq [m+1,N]$.
\end{enumerate}
As such, the superpartition $\Omega$ satisfies the conditions of
 Theorem~\ref{theoPieri} (with 
 $C_{J,\sigma'}(x) = C_{J,\sigma}(x)$).
\end{lemma}

\begin{proof} 
  We first show that $(\Omega,w\sigma')$ is a superevaluation.  By definition,  we have to show that  $w\sigma'\Omega=\Omega^*$ and that  $w\sigma'\bigl(\Omega + (1^m)\bigr)=\Omega^\circledast$ for a certain $s\in \mathfrak S_m \times \mathfrak S_{m+1,N}$.  We will show, equivalently, that
  $(w \sigma')^{-1} \Omega^{*} = \Omega$ and that
$(w \sigma')^{-1} \Omega^{\circledast} = \Omega + (1^m)$.  Observe that
  $$w \sigma' = w \sigma s= w w^{-1} \tilde{\sigma} w s= \tilde{\sigma} w s$$
It thus suffices to show that
$s^{-1} w^{-1} \tilde{\sigma}^{-1} \Omega^*=\Omega$
and $s^{-1} w^{-1} \tilde{\sigma}^{-1} \Omega^\circledast=\Omega + (1^m)$.
  From the definition of 
$\tilde{\sigma}$, it is immediate that $\tilde{\sigma}^{-1}$ also interchanges
  the $\BCC$-rows and the $\OSS$-rows. Hence $\tilde{\sigma}^{-1} \Omega^\circledast/\tilde{\sigma}^{-1} \Omega^*=\Lambda^\circledast/\Lambda^*$. Since by definition
$w^{-1}$ sends $\Lambda^\circledast/\Lambda^*$ to $[m]$, we have that 
  $w^{-1} \tilde{\sigma}^{-1}$ sends the rows in the diagram of $\Omega$ ending with a circle to $[m]$, that is, $w^{-1} \tilde{\sigma}^{-1} \Omega^*=\pmb v$
  and  $w^{-1} \tilde{\sigma}^{-1} \Omega^\circledast=\pmb v + (1^m)$ for a certain $v \in \mathbb Z_{\geq 0}^N$. Using any
   $s^{-1} \in \mathfrak S_m \times \mathfrak S_{m+1,N}$ such that $s^{-1} \pmb v=\Omega$, we obtain that $s^{-1} w^{-1} \tilde{\sigma}^{-1} \Omega^*=\Omega$
and $s^{-1} w^{-1} \tilde{\sigma}^{-1} \Omega^\circledast=\Omega + (1^m)$ as wanted.
  We will take this as the definition of $s$ in the rest of the proof.

Observe that {\it 1)} is equivalent to 
$$w \sigma'(\Omega+(1^m))=\tau_{w(J)}w(\Lambda+(1^m))=\tau_{\tilde \sigma (\tilde J)}\Lambda^\circledast$$
by definition of $w$ and $\tilde J$.  Since we have shown that  {\it 3)} holds, we only have left to show that
$\Omega^{\circledast} =\tau_{\tilde{\sigma}(\tilde{J})}(\Lambda^{\circledast})$. But by definition of $\tilde J$ and $\tilde \sigma$, the set $\tilde \sigma(\tilde J)$ corresponds to the  $\BCC$-rows and $\BSS$-rows in the diagram of $\Omega/\Lambda$, that is, to the rows of $\Omega^{\circledast}/\Lambda^{\circledast}$.  We have thus shown that $\Omega^{\circledast} =\tau_{\tilde{\sigma}(\tilde{J})}(\Lambda^{\circledast})$.

As for {\it 2)}, let  $x \in \sigma'([m]) \cap L =\sigma([m]) \cap L$.
Therefore, $w(x) \in w \sigma([m]) \cap w (L)= \tilde \sigma w([m]) \cap w(L)$.
Now, $w([m])$ corresponds to the $\OSS$-rows and the $\WCC$-rows in the diagram of $\Omega/\Lambda$, which implies that  $\tilde \sigma w([m])$ corresponds to the  $\WCC$-rows and  $\BCC$-rows in that diagram. Since $w([m]) \cap w(L)=\emptyset$, $w(L)$ cannot correspond to any $\OSS$-row or any $\WCC$-row.  Therefore,
$w(x) \in  \tilde \sigma w([m]) \cap w(L)$ needs to correspond to a  $\BCC$-row.
But this is impossible because  $w(J) \cap w(L)=\emptyset$ and
the $\BCC$-rows belong to $w(J)=\tilde \sigma (\tilde J)$.

Finally, we have to show {\it 4)}.  By definition of a vertical strip of type I, the $\BCC$, $\BSS$,  $\WCC$ and $\OSS$ rows are all distinct.  Now, $\tilde{\sigma} \tilde{J}$ corresponds to the $\BCC$-rows and the $\BSS$-rows, while
 $w([m])$ corresponds to the $\WCC$-rows and the $\OSS$-rows. Hence, 
$\tilde{\sigma} \tilde{J} \subseteq w([m+1,N])$, which implies that
$J = w^{-1} \tilde{\sigma} \tilde{J} \subseteq [m+1,N]$.     
\end{proof}

\begin{example}
Consider the following vertical strip of type I:
\begin{equation}
{\small \tableau[scY]{  &&&& \fl \\ &&& \bl \cerclep \\  && \tf \ocircle \\  & \\  &\tf \ocircle  \\ \bl \cerclep}}
\end{equation}
We have in this case that $\tilde{J} = \{1,3,5\}$.  Taking $\tilde{\sigma} =
[1, 3, 2, 4, 6, 5]$, and 
$w = [3, 5, 1, 2, 4, 6]$ (in one-line notation), we get that $J= \{3,4,6\}$ and $\sigma = [4, 6, 3, 1, 5, 2]$. 
\end{example}
Using Lemma~\ref{cond2vert} and Lemma~\ref{vert2cond}, we can rewrite Theorem~\ref{theoPieri} in a more precise fashion.
\begin{corollary} \label{coroPieri} For $r\in \{1,\dots,N-m\}$, the bisymmetric Macdonald polynomial
$\mathcal P_{\Lambda}(x;q,t)$ obeys the following Pieri rules
\begin{equation*}
  e_{r}(x_{m+1},\dots,x_N)    {\mathcal P}_{\Lambda}(x;q,t) = \sum_\Omega \left(\frac{u_\Lambda^+(C_{J,\sigma})}{ u_\Lambda^+(\Delta_m^t)} \frac{u_{\Lambda_0^+}(\mathcal P_\Lambda)}{u_{\Lambda_0^+}(\mathcal P_\Omega)} \right)   {\mathcal P}_{\Omega}(x,q,t)  
\end{equation*}
where the sum is over all superpartitions $\Omega$ such that $\Omega/\Lambda$ is a vertical $r$-strip of type I.  Note that $C_{J,\sigma}(x)$ was defined in 
Proposition~\ref{Op+}, where 
$\sigma$ and $J$ can be obtained in the following manner from  the diagram of $\Omega/\Lambda$: let $\tilde \sigma$ be any permutation that interchanges the
  $\OSS$-rows and the  $\BCC$-rows while leaving the remaining rows invariant, and let  $\tilde J$ be the set of
  $\OSS$-rows and $\BSS$-rows.  Then
$$\sigma = w^{-1} \tilde{\sigma} w  \quad {\rm   and   } \quad J = w^{-1} \tilde{\sigma} (\tilde{J})$$
where $w$ is such that $(\Lambda,w)$ is a superevaluation.
\end{corollary}

\begin{example}
  The superpartitions  that appear in the expansion of the multiplication of
$e_2(x_3,x_4,\dots,x_N)$ and $\mathcal P_{(2,0;1)}(x_1,\dots,x_N;q,t) $ in terms of bisymmetric Macdonald polynomials are:
\begin{equation*}
  \small \tableau[scY]{ & & \tf \ocircle \\   & \bl \cerclep \\  \tf \ocircle \\ \bl \cerclep  } \qquad  \qquad \small \tableau[scY]{ & &  \bl \cercle{} \\ & \fl \\ \tf \ocircle \\ \bl \cerclep  } \qquad \qquad \small \tableau[scY]{ & & \bl \cercle{} \\ & \bl \cerclep \\ \tf \ocircle \\ \fl   } \qquad  \qquad \small \tableau[scY]{ & & \bl \cercle{} \\ \\ \tf \ocircle \\ \fl \\ \bl \cerclep  } 
\end{equation*}
\end{example}
To be more precise, we have that
\begin{align*}
&e_2(x_3,x_4,\dots,x_N) \mathcal P_{(2,0;1)}
=    \dfrac{q(1-t)}{1-qt}  \mathcal P_{(1,0;3,1)} + \dfrac{(1-q)(1-qt^2)}{(1-qt)^2} \mathcal P_{(2,0;2,1)} \\
&\qquad \qquad \qquad \qquad \qquad- \dfrac{(t+1)(1-t)(1-q)(1-q^2t^4)}{(1-q^2t^3)(1-qt)(1-qt^2)}  \mathcal P_{(2,1;1,1)} + \dfrac{(1-qt)(1-t^3)}{(1-t)(1-qt^3)}  \mathcal P_{(2,0;1,1,1)} 
\end{align*}
We now give more details on how the coefficient of $\mathcal P_{(1,0;3,1)}$ for instance was obtained. 
The diagram $\Omega/ \Lambda$ is in this case
\begin{equation*}
  \small \tableau[scY]{ & & \tf \ocircle \\   & \bl \cerclep \\  \tf \ocircle \\ \bl \cerclep  }  
\end{equation*}
Choosing instance $\tilde{\sigma}=(1 2) (3 4)$,  $\widetilde{J} = \{1,3\}$ and $w= (2 3)$, we then obtain from Corollary~\ref{coroPieri} that
$$\sigma = (23) (12)(34)(23) = [3,4,1,2]\quad \text{and} \quad J = (23)(12)(34)\{1,3\} = \{3,4\}$$
Lemma~\ref{RelDet} gives 
$$\Phi(\sigma) = -(t-1)^{2}x_3x_4\dfrac{(x_1-x_2)(x_3-x_4)}{(x_1-x_3)(x_1-x_4)(x_2-x_3)(x_2-x_4)}$$
while Theorem~\ref{theoPieri} with $m=2$ and $r=2$  yields
$$C_{J,\sigma} = t^{3-2N}\dfrac{tx_1-x_2}{x_1-x_2}\tau_3\tau_4\left(\dfrac{tx_3-x_4}{x_3-x_4}\dfrac{tx_4-x_3}{x_4-x_3}\Phi(\sigma) K_{13}K_{24} (x_1-x_2)  \right)$$
Taking 
$$\frac{u_\Lambda^+(C_{J,\sigma})}{ u_\Lambda^+(\Delta_m^t)} \frac{u_{\Lambda_0^+}(\mathcal P_\Lambda)}{u_{\Lambda_0^+}(\mathcal P_\Omega)}$$
we finally get the desired coefficient.

\section{Pieri rules for $e_r(x_1,\dots,x_m)$}

In Theorem~\ref{theoPieri} and Corollary~\ref{coroPieri}, we obtained Pieri rules for the action of $e_r(x_{m+1},\dots,x_N)$ on  bisymmetric Macdonald polynomials.  In this section, we will present Pieri rules for the the action of $e_r(x_{1},\dots,x_m)$.  Note that by the fundamental theorem of symmetric functions, any bisymmetric polynomial can be written as a linear combination of products of the $e_r(x_{1},\dots,x_m)$'s for $r=1,\dots,m$ and the $e_k(x_{m+1},\dots,x_N)$'s for $k=1,\dots,N-m$.  As such,  Pieri rules for $e_r(x_{1},\dots,x_m)$
and  $e_r(x_{m+1},\dots,x_N)$ provide a full set of Pieri rules.

Since the proof in the $e_r(x_{1},\dots,x_m)$ case is quite similar to that in the $e_r(x_{m+1},\dots,x_N)$ case, we will only state the main results.  Details can be found in \cite{Con}.

The analog of Proposition~\ref{Op+} that we need is the following.
\begin{proposition} \label{Op-} Let $f(x)$ be any bisymmetric function. For any $1 \leq r \leq m$, we have that
  $$e_{r}(Y_1,\ldots,Y_m) \Delta_m^t(x) f(x) =  \sum_{\substack{ J \subseteq [N] \\ |J| = r }} \, \,  \sum_{\substack{ [\sigma] \in  \mathfrak S_N/ ( \mathfrak S_m \times  \mathfrak S_{m+1,N}) \\ \sigma([m]) \cap L= \emptyset, \,   \sigma(J) \subseteq [m] }} 
  D_{J,\sigma}(x) \tau_{J} K_\sigma f(x)$$
where the coefficient $D_{J,\sigma}(x)$ is given by
$$D_{J,\sigma}(x) =  (-1)^{ \# \BB^c_{\sigma}} t^{(r+1-2N)r/2} A_{m}(x) A_{J \times [m+1,N]}(x,x) \varPhi(\sigma) \tau_J \bigl( A_{J \times \AA_{\sigma}}(x,x)K_{\sigma} \Delta_m(x) \bigr)
$$ 
\end{proposition}
Using a version of  Lemma~\ref{lemmaOmega} (where $C_{J,\sigma}(x)$ is replaced by $D_{J,\sigma}$) and \eqref{aut3}, which can be used since  $e_r(x_1,\dots,x_m)$ is symmetric in the variables $x_1,\dots, x_m$ and of homogeneous degree $r$, we get the desired Pieri rules.   
\begin{theorem} \label{PieriAlg-} For any  $1 \leq r\leq m$, the bisymmetric Macdonald polynomial $\mathcal P_{\Lambda}(x;q,t)$ is such that
\begin{equation*}
  e_{r}(x_{1},\dots,x_m)    {\mathcal P}_{\Lambda}(x;q,t) =q^r \sum_\Omega  \left(\frac{u_\Lambda^+(D_{J,\sigma})}{ u_\Lambda^+(\Delta_m^t)} \frac{u_{\Lambda_0^+}(\mathcal P_\Lambda)}{u_{\Lambda_0^+}(\mathcal P_\Omega)} \right)   {\mathcal P}_{\Omega}(x,q,t)  
\end{equation*}
where the coefficients $D_{J,\sigma}(x)$ were obtained explicitly in Proposition~\ref{Op-} and 
where the sum is over all superpartitions $\Omega$ such that there exists a
$\sigma \in \mathfrak S_N$ and
a $J \subseteq [N]$ of size $r$
such that
\begin{itemize}
\item  $\sigma(\Omega+(1^m))=\tau_J(\Lambda+(1^m))$ 
\item $\sigma([m]) \cap L = \emptyset$, where $L=[m+1,\dots,N]\setminus J$
\item $\sigma (J) \subseteq [m]$  
\item  $(\Omega,w\sigma)$ is a superevaluation, where $w$ is such that $(\Lambda,w)$ generates a superevaluation
\end{itemize}
\end{theorem}
The $\Omega$'s that appear in the Pieri rules of Theorem~\ref{PieriAlg-} are also special vertical $r$-strips.
 \begin{definition}  We will say that 
    $\Omega/ \Lambda$ is a vertical $r$-strip of type II if
\begin{enumerate}
   \item  $\Omega/ \Lambda$ is a vertical $r$-strip
   \item there are no
     $\BSS$-rows in the diagram of   $\Omega/ \Lambda$
\end{enumerate}
  \end{definition}
For instance, if $\Lambda=(5,3,1;4,3)$ and $\Omega=(5,2,0;4,4,3)$ then 
   $\Omega/ \Lambda$ is a vertical 2-strip of type II.
\begin{equation}
{\small \tableau[scY]{  &&&& & \bl \cercle{} \\  &&&  \\   &&&  \tf \ocircle \\  &&  \\  & \tf \ocircle & \bl \cerclep  \\   \bl \cerclep \\ }}
\end{equation}
As was the case in Corollary~\ref{coroPieri}, we can rewrite Theorem~\ref{PieriAlg-} in a more precise way using vertical $r$-strips of type II.

\begin{corollary} \label{coroPieri2} For $r\in \{1,\dots,N-m\}$, the bisymmetric Macdonald polynomial
$\mathcal P_{\Lambda}(x;q,t)$ obeys the following Pieri rules
\begin{equation*}
  e_{r}(x_{1},\dots,x_m)    {\mathcal P}_{\Lambda}(x;q,t) =q^r \sum_\Omega  \left(\frac{u_\Lambda^+(D_{J,\sigma})}{ u_\Lambda^+(\Delta_m^t)} \frac{u_{\Lambda_0^+}(\mathcal P_\Lambda)}{u_{\Lambda_0^+}(\mathcal P_\Omega)} \right)   {\mathcal P}_{\Omega}(x,q,t)  
\end{equation*}
where the sum is over all superpartitions $\Omega$ such that $\Omega/\Lambda$ is a vertical $r$-strip of type II.  Note that $D_{J,\sigma}(x)$ was defined in 
Proposition~\ref{Op-}, where 
$\sigma$ and $J$ can be obtained in the following manner from  the diagram of $\Omega/\Lambda$:  
 let $\tilde \sigma$ be any permutation that interchanges the
  $\OSS$-rows and the  $\BCC$-rows, while leaving the remaining rows invariant (including the ${\footnotesize\tableau[scY]{ \tf \ocircle & \bl \cerclep} }$ rows), and let  $\tilde J$ be the set of
  $\OSS$-rows.  Then
$$\sigma = w^{-1} \tilde{\sigma} w  \quad {\rm   and   } \quad J = w^{-1} \tilde{\sigma} (\tilde{J})$$
where $w$ is such that $(\Lambda,w)$ is a superevaluation. 
\end{corollary}

\begin{example}
The superpartitions that appear in the  expansion of the  multiplication of $e_2(x_1,x_2)$ and $P_{(2,0;1)}(x;q,t) $ are given by:
\begin{equation*}
 \small \tableau[scY]{ & & \tf \ocircle & \bl \cerclep \\   & \bl \cerclep \\  \tf \ocircle  } \qquad \small \tableau[scY]{ & & \tf \ocircle & \bl \cerclep \\   \\ \tf \ocircle \\ \bl \cerclep  } \qquad  \small \tableau[scY]{ & & \tf \ocircle \\  & \bl \cerclep \\ \tf \ocircle \\ \bl \cerclep  } 
\end{equation*}
To be more precise, we have that
\begin{align*}
e_2(x_1,x_2) 
P_{(2,0;1)}=
   \mathcal P_{(3,1;1)}
- \dfrac{q(1+t)(1-t)}{1-qt^2} \mathcal  P_{(3,0;1,1)} + \dfrac{q^3(1-t)(1-q^2t^3)(1-qt)}{(1-qt)(1-q^3t^3)(1-q^2t)} \mathcal  P_{(1,0;3,1)}
\end{align*}
\end{example}

\end{document}